\documentclass[times]{speauth}

\usepackage{moreverb}

\usepackage{amsfonts,amssymb,amsthm,newlfont,graphicx,subfigure,algorithmic,algorithm,fixltx2e,booktabs,bm}
\usepackage{epstopdf}
\usepackage{empheq} 
\usepackage[colorlinks,bookmarksopen,bookmarksnumbered,citecolor=red,urlcolor=red]{hyperref}
\numberwithin{algorithm}{section}

\newcommand\BibTeX{{\rmfamily B\kern-.05em \textsc{i\kern-.025em b}\kern-.08em
T\kern-.1667em\lower.7ex\hbox{E}\kern-.125emX}}

\newtheorem{thm}{Theorem}[section]

\newtheorem{rem}{Remark}[section]

\numberwithin{equation}{section}
\renewcommand{\theequation}{\thesection.\arabic{equation}}
\def\simgt{\,\hbox{\lower0.6ex\hbox{$>$}\llap{\raise0.3ex\hbox{$\sim$}}}\,}
\def\simlt{\,\hbox{\lower0.6ex\hbox{$<$}\llap{\raise0.3ex\hbox{$\sim$}}}\,}
\def\simgteq{\,\hbox{\lower0.6ex\hbox{$\ge$}\llap{\raise0.6ex\hbox{$\sim$}}}\,}
\def\simlteq{\,\hbox{\lower0.6ex\hbox{$\le$}\llap{\raise0.6ex\hbox{$\sim$}}}\,}
\makeatletter
\def\user@resume{resume}
\def\user@intermezzo{intermezzo}
\newcounter{previousequation}
\newcounter{lastsubequation}
\newcounter{savedparentequation}
\renewenvironment{subequations}[1][]{%
      \def\user@decides{#1}%
      \setcounter{previousequation}{\value{equation}}%
      \ifx\user@decides\user@resume
           \setcounter{equation}{\value{savedparentequation}}%
      \else
      \ifx\user@decides\user@intermezzo
           \refstepcounter{equation}%
      \else
           \setcounter{lastsubequation}{0}%
           \refstepcounter{equation}%
      \fi\fi
      \protected@edef\theHparentequation{%
          \@ifundefined {theHequation}\theequation \theHequation}%
      \protected@edef\theparentequation{\theequation}%
      \setcounter{parentequation}{\value{equation}}%
      \ifx\user@decides\user@resume
           \setcounter{equation}{\value{lastsubequation}}%
         \else
           \setcounter{equation}{0}%
      \fi
      \def\theequation  {\theparentequation  \alph{equation}}%
      \def\theHequation {\theHparentequation \alph{equation}}%
      \ignorespaces
}{%
  \ifx\user@decides\user@resume
       \setcounter{lastsubequation}{\value{equation}}%
       \setcounter{equation}{\value{previousequation}}%
  \else
  \ifx\user@decides\user@intermezzo
       \setcounter{equation}{\value{parentequation}}%
  \else
       \setcounter{lastsubequation}{\value{equation}}%
       \setcounter{savedparentequation}{\value{parentequation}}%
       \setcounter{equation}{\value{parentequation}}%
  \fi\fi
  \ignorespacesafterend
}
\makeatother
\makeatletter
\newcommand{\algorithmicbreak}{\textbf{break}}
\newcommand{\BREAK}{\STATE \algorithmicbreak}
\newcommand{\algorithmiccontinue}{\textbf{continue}}
\newcommand{\CONTINUE}{\STATE \algorithmiccontinue}
\makeatother
\begin{document}

\runningheads{Kareem T. Elgindy}{Optimization via Chebyshev Polynomials}

\title{Optimization via Chebyshev Polynomials}

\author{Kareem T. Elgindy\corrauth}

\address{Mathematics Department, Faculty of Science, Assiut University, Assiut 71516, Egypt}

\corraddr{Mathematics Department, Faculty of Science, Assiut University, Assiut 71516, Egypt}

\begin{abstract}
This paper presents for the first time a robust exact line-search method based on a full pseudospectral (PS) numerical scheme employing orthogonal polynomials. The proposed method takes on an adaptive search procedure and combines the superior accuracy of Chebyshev PS approximations with the high-order approximations obtained through Chebyshev PS differentiation matrices (CPSDMs). In addition, the method exhibits quadratic convergence rate by enforcing an adaptive Newton search iterative scheme. A rigorous error analysis of the proposed method is presented along with a detailed set of pseudocodes for the established computational algorithms. Several numerical experiments are conducted on one- and multi-dimensional optimization test problems to illustrate the advantages of the proposed strategy. 
\end{abstract}


\maketitle

\vspace{-6pt}

\section{Introduction}
\label{int}
The area of optimization received enormous attention in recent years due to the rapid progress in computer technology, development of user-friendly software, the advances in scientific computing provided, and most of all, the remarkable interference of mathematical programming in crucial decision-making problems. One-dimensional optimization or simply line search optimization is a branch of optimization that is most indispensable, as it forms the backbone of nonlinear programming algorithms. In particular, it is typical to perform line search optimization in each stage of multivariate algorithms to determine the best length along a certain search direction; thus, the efficiency of multivariate algorithms largely depends on it. Even in constructing high-order numerical quadratures, line search optimization emerges in minimizing their truncation errors; thus boosting their accuracy and allowing to obtain very accurate solutions to intricate boundary-value problems, integral equations, integro-differential equations, and optimal control problems in short times via stable and efficient numerical schemes; cf. \cite{Elgindy2012c,Elgindy2013a,Elgindy2012d,Elgindy2013,Elgindy2016b,Elgindy2015}. 

Many line search methods were presented in the literature in the past decades. Some of the most popular line search methods include interpolation methods, Fibonacci's method, golden section search method, secant method, Newton's method, to mention a few; cf. \cite{Pedregal2006,Chong2013,Nocedal2006}. Perhaps Brent's method is considered one of the most popular and widely used line search methods nowadays. It is a robust version of the inverse parabolic interpolation method that makes the best use of both techniques, the inverse parabolic interpolation method and the golden section search method, and can be implemented in MATLAB software, for instance, using the `fminbnd' optimization solver. The method is a robust optimization algorithm that does not require derivatives; yet it lacks the rapid convergence rate manifested in the derivative methods when they generally converge. 

In $2008$, \cite{Elgindy2008a} gave a new approach for constructing an exact line search method using Chebyshev polynomials, which have become increasingly important in scientific computing, from both theoretical and practical points of view. They introduced a fast line search method that is an adaptive version of Newton's method. In particular, their method forces Newton's iterative scheme to progress only in descent directions. Moreover, the method replaces the classical finite-difference formulas for approximating the derivatives of the objective function by the more accurate step lengths along the Chebyshev pseudospectral (PS) differentiation matrices (CPSDMs); thus getting rid of the dependency of the iterative scheme on the choice of the step-size, which can significantly affect the quality of calculated derivatives approximations. Although the method worked quite well on some test functions, where classical Newton's method fail to converge, the method may still suffer from some drawbacks that we discuss thoroughly later in the next section. 

In this article, the question of how to construct an exact line search method based on PS methods is re-investigated and explored. We present for the first time a robust exact line-search method based on a full PS numerical scheme employing Chebyshev polynomials and adopting two global strategies: (i) Approximating the objective function by an accurate fourth-order Chebyshev interpolant based on Chebyshev-Gauss-Lobatto (CGL) points; (ii) approximating the derivatives of the function using CPSDMs. The first strategy actually plays a significant role in capturing a close profile to the objective function from which we can determine a close initial guess to the local minimum, since Chebyshev polynomials as basis functions can represent smooth functions to arbitrarily high accuracy by retaining a finite number of terms. This approach also improves the performance of the adaptive Newton's iterative scheme by starting from a sufficiently close estimate; thus moving in a quadratic convergence rate. The second strategy yields accurate search directions by maintaining very accurate derivative approximations. The proposed method is also adaptive in the sense of searching only along descent directions, and avoids the raised drawbacks pertaining to the \cite{Elgindy2008a} method. 

The rest of the article is organized as follows: In the next section, we revisit the \cite{Elgindy2008a} method highlighting its strengths aspects and weaknesses. In Section \ref{sec:TSCPDM1}, we provide a modified explicit expression for higher-order CPSDMs based on the successive differentiation of the Chebyshev interpolant. We present the novel line search strategy in Section \ref{sec:PLSM} using first-order and/or second-order information, and discuss its integration with multivariate nonlinear optimization algorithms in Section \ref{eq:IWMNOA1}. Furthermore, we provide a rigorous error and sensitivity analysis in Sections \ref{sec:EA1} and \ref{sec:SA1}, respectively. Section \ref{sec:numex1} verifies the accuracy and efficiency of the proposed method through an extensive sets of one- and multi-dimensional optimization test examples followed by some conclusions given in Section \ref{Conc}. Some useful background in addition to some useful pseudocodes for implementing the developed method are presented in Appendices \ref{sec:P1} and \ref{appendix:PS1}, respectively.
\section{The \texorpdfstring{\cite{Elgindy2008a}}{Elgindy and Hedar (2008)} Optimization Method Revisited}
\label{sec:TEAHOMR}
The \cite{Elgindy2008a} line search method is an adaptive method that is considered an improvement over the standard Newton's method and the secant method by forcing the iterative scheme to move in descent directions. The developed algorithm is considered unique as it exploits the peculiar convergence properties of spectral methods, the robustness of orthogonal polynomials, and the concept of PS differentiation matrices for the first time in a one-dimensional search optimization. In addition, the method avoids the use of classical finite difference formulas for approximating the derivatives of the objective function, which are often sensitive to the values of the step-size. 

The success in carrying out the above technique lies very much in the linearity property of the interpolation, differentiation, and evaluation operations. Indeed, since all of these operations are linear, the process of obtaining approximations to the values of the derivative of a function at the interpolation points can be expressed as a matrix-vector multiplication. In particular, if we approximate the derivatives of a function $f(x)$ by interpolating the function with an $n$th-degree polynomial $P_n(x)$ at, say, the CGL points defined by Eq. \eqref{eq:app:CGL1}, then the values of the derivative $P'_n(x)$ at the same $(n + 1)$ points can be expressed as a fixed linear combination of the given function values, and the whole relationship can be written in the following matrix form:
\begin{equation}\label{eq:dmat}
\left( {\begin{array}{*{20}{l}}
{{P'_n}({x_0})}\\
{\,\,\quad  \vdots }\\
{{P'_n}({x_n})}
\end{array}} \right) = \left( {\begin{array}{*{20}{c}}
{d_{00}^{(1)}}& \ldots &{d_{0n}^{(1)}}\\
 \vdots & \ddots & \vdots \\
{d_{n0}^{(1)}}& \cdots &{d_{nn}^{(1)}}
\end{array}} \right)\left( {\begin{array}{*{20}{l}}
{f({x_0})}\\
{\quad \, \vdots }\\
{f({x_n})}
\end{array}} \right).
\end{equation}
Setting $\bm{\mathcal{F}} = [f(x_0 ),f(x_1 ), \ldots ,f(x_n )]^T$, as the vector consisting of values of $f(x)$ at the $(n + 1)$ interpolation points, $\bm{\Psi} = [P'_n(x_0 ),P'_n(x_1 ), \ldots ,P'_n(x_n )]^T$, as the values of the derivative at the CGL points, and $\bm{D}^{(1)} = (d_{ij}^{(1)})\;,\;0 \le i,j \le n$, as the first-order differentiation matrix mapping $\bm{\mathcal{F}} \to \bm{\Psi}$, then Formula \eqref{eq:dmat} can be written in the following simple form
\begin{equation}\label{eq:dmat2}
    \bm{\Psi} = \bm{D}^{(1)} \bm{\mathcal{F}}.
\end{equation}
Eq. \eqref{eq:dmat2} is generally known as the PS differentiation rule, and it generally delivers better accuracy than standard finite-difference rules. We show later in the next section some modified formulas for calculating the elements of a general $q$th-order CPSDM, $\bm{D}^{(q)}\; \forall q \ge 1$. 
Further information on classical explicit expressions of the entries of differentiation matrices can be found in \cite{Weideman2000,Baltensperger2000,Costa2000,Elbarbary2005}.

A schematic figure showing the framework of \cite{Elgindy2008a} line search method is shown in Figure \ref{fig:ls2}. As can be observed from the figure, the method starts by transforming the uncertainty interval $[a,b]$ into the domain $[-1,1]$ to exploit the rapid convergence properties provided by Chebyshev polynomials. The method then takes on a global approach in calculating the derivatives of the function at the CGL points, $\{x_i\}_{i=0}^n$, using CPSDMs. If the optimality conditions are not satisfied at any of these candidate points, the method then endeavors to locate the best point $x_m$ that minimizes the function, and decides whether to keep or expand the uncertainty interval. An update $\tilde x_m$ is then calculated using the descent direction property followed by updating row $m$ in both matrices $\bm{D}^{(1)}$ and $\bm{D}^{(2)}$. The iterative scheme proceeds repeatedly until the stopping criterion is satisfied.

Although the method possesses many useful features over classical Newton's method and the secant method; cf. \cite[Section 8]{Elgindy2008a}, it may still suffer from the following drawbacks: (i) In spite of the fact that, $\bm{D}^{(1)}$ and $\bm{D}^{(2)}$ are constant matrices, the method requires the calculation of their entire elements beforehand. We show later that we can establish rapid convergence rates using only one row from each matrix in each iterate. (ii) The method attempts to calculate the first and second derivatives of the function at each point $x_i \in [-1, 1], i =0, \ldots, n$. This could be expensive for large values of $n$, especially if the local minimum $t^*$ is located outside the initial uncertainty interval $[a, b]$; cf. \cite[Table 6]{Elgindy2008a}, for instance. (iii) The method takes on a global approach for approximating the derivatives of the objective function using CPSDMs instead of the usual finite-difference formulas that are highly sensitive to the values of the step-size. However, the method locates a starting approximation by distributing the CGL points along the interval $[-1,1]$, and finding the point that best minimizes the value of the function among all other candidate points. We show later in Section \ref{sec:PLSM} that we can significantly speed up this process by adopting another global approach based on approximating the function via a fourth-order Chebyshev interpolant, and determining the starting approximation through finding the best root of the interpolant derivative using exact formulas. (iv) During the implementation of the method, the new approximation to the local minimum, $\tilde x_m$, may lie outside the interval $[-1,1]$ in some unpleasant occasions; thus the updated differentiation matrices may produce false approximations to the derivatives of the objective function. (v) To maintain the adaptivity of the method, the authors proposed to flip the sign of the second derivative $f''$ whenever $f'' < 0$, at some point followed by its multiplication with a random positive number $\beta$. This operation may not be convenient in practice; therefore, we need a more efficient approach to overcome this difficulty. (vi) Even if $f'' > 0$, at some point, the magnitudes of $f'$ and $f''$ may be too small slowing down the convergence rate of the method. This could happen for instance if the function has a multiple local minimum, or has a nearly flat profile about $t^*$. (vii) Suppose that $t^*$ belongs to one side of the real line while the initial search interval $[a, b]$ is on the other side. According to the presented method, the search interval must shrink until it converges to the point zero, and the search procedure halts. Such drawbacks motivate us to develop a more robust and efficient line search method. 
\begin{figure}[ht]
\centering
\includegraphics[scale=0.45]{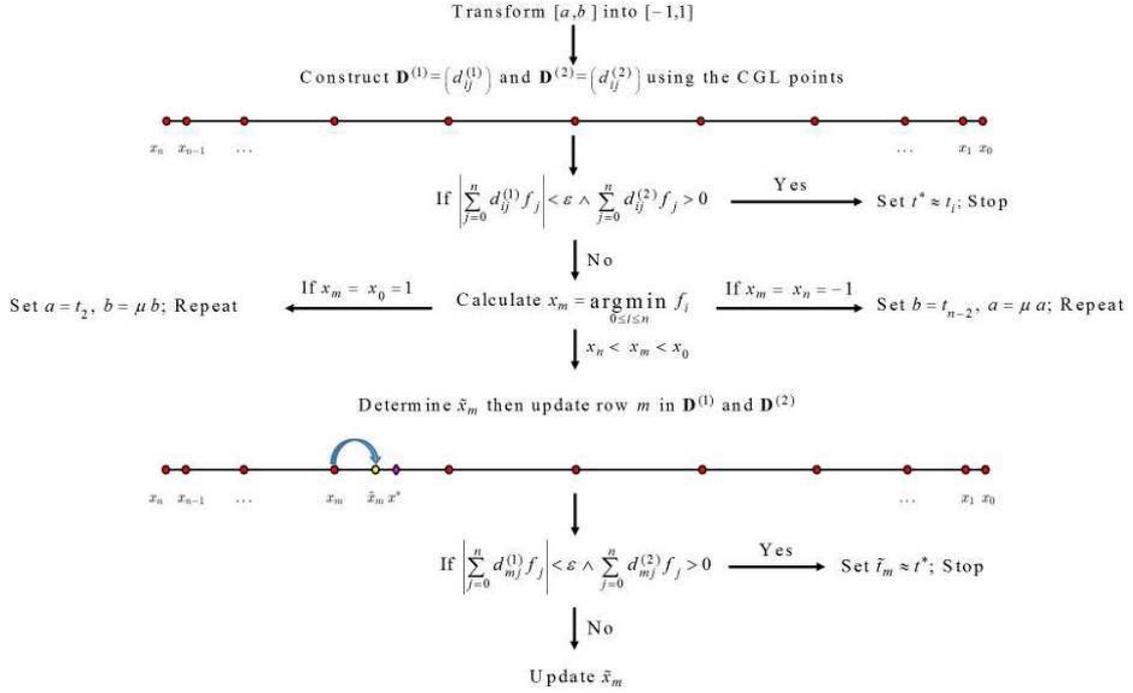}\\
\caption{An illustrative figure showing the framework of the line search method introduced by \cite{Elgindy2008a} for minimizing a single-variable function $f: [a, b] \to \mathbb{R}$, where $\{x_i\}_{i=0}^n$ are the CGL points, $t_i = \left((b - a) x_i + a + b\right)/2\; \forall i$, are the corresponding points in $[a, b]$, $f_i = f(t_i)\; \forall i$, $\varepsilon$ is a relatively small positive number, $\mu > 1$, is a parameter preferably chosen as $1.618^k$, where $\rho \approx 1.618$ is the golden ratio, and $k = 1,2, \ldots$, is the iteration counter}
\label{fig:ls2}
\end{figure}
\section{High-Order CPSDMs}
\label{sec:TSCPDM1}
In the sequel, we derive a modified explicit expression for higher-order CPSDMs to that stated in \cite{Elgindy2008a} based on the successive differentiation of the Chebyshev interpolant. The higher derivatives of Chebyshev polynomials are expressed in standard polynomial form rather than as Chebyshev polynomial series expansion. The relation between Chebyshev polynomials and trigonometric functions is used to simplify the expressions obtained, and the periodicity of the cosine function is used to express it in terms of existing nodes. 

The following theorem gives rise to a modified useful form for evaluating the $m$th-derivative of Chebyshev polynomials.

\begin{thm}\label{thm1}
The $m$th-derivative of the Chebyshev polynomials is given by
\begin{subequations}
\begin{empheq}[left={T_k^{(m)}(x) =}\empheqlbrace]{align}
&0,\quad 0 \le k < m,\label{eq:dercheb5}\\
&1,\quad k = m = 0,\label{eq:dercheb3}\\
&{2^{k - 1}}m!,\quad k = m \ge 1,\label{eq:dercheb4}\\
&\sum\limits_{l = 0}^{\left\lfloor k/2 \right\rfloor} {\gamma _{l,k}^{(m)}\,c_l^{(k)}{x^{k - 2l - m}}} ,\quad k > m \ge 0 \wedge x \ne 0,\label{eq:dercheb1}\\
&\beta _k^{(m)}\cos \left( {\frac{\pi }{2}\left( {k - {\delta _{\frac{{m + 1}}{2},\left\lfloor {\frac{{m + 1}}{2}} \right\rfloor }}} \right)} \right),\quad k > m \ge 0 \wedge x = 0,\label{eq:dercheb2}
\end{empheq}
\end{subequations}
where
\begin{subequations}
\begin{empheq}[left={\gamma _{l,k}^{(m)} =}\empheqlbrace]{align}
&1,\quad m = 0,\\
&\left( {k - 2l - m + 1} \right){\left( {k - 2l - m + 2} \right)_{m - 1}},\quad m \ge 1,
\end{empheq}
\end{subequations}
\begin{subequations}\label{eq:clkmain1}
\begin{empheq}[left={c_l^{(k)} =}\empheqlbrace]{align}
&1,\quad l = k = 0,\\
&{2^{k - 1}},\quad l = 0 \wedge k \ge 1,\\
&- \frac{{\left( {k - 2l + 1} \right)\left( {k - 2l + 2} \right)}}{{4l\left( {k - l} \right)}} c_{l - 1}^{(k)},\quad l = 1, \ldots ,\left\lfloor {k/2} \right\rfloor  \wedge k \ge 1,
\end{empheq}
\end{subequations}
\small{
\begin{subequations}
\begin{empheq}[left={\beta _k^{(m)} =}\empheqlbrace]{align}
&1,\quad m = 0,\\
&{\left( { - 4} \right)^{\left\lfloor {\frac{{m - 1}}{2}} \right\rfloor }}{\left( { - 1} \right)^{{\delta _{\frac{{m + 1}}{2},\left\lfloor {\frac{{m + 1}}{2}} \right\rfloor }} + \left\lfloor {\frac{{m + 1}}{2}} \right\rfloor }}{k^{{\delta _{\frac{m}{2},\left\lfloor {\frac{m}{2}} \right\rfloor }} + 1}}{\left( {\frac{1}{2}\left( { - k - {\delta _{\frac{{m + 1}}{2},\left\lfloor {\frac{{m + 1}}{2}} \right\rfloor }} + 2} \right)} \right)_{\left\lfloor {\frac{{m - 1}}{2}} \right\rfloor }} \times \nonumber\\
&{\left( {\frac{1}{2}\left( {k - {\delta _{\frac{{m + 1}}{2},\left\lfloor {\frac{{m + 1}}{2}} \right\rfloor }} + 2} \right)} \right)_{\left\lfloor {\frac{{m - 1}}{2}} \right\rfloor }},\quad m \ge 1,
\end{empheq}
\end{subequations}
}
$\left\lfloor x \right\rfloor$ is the floor function of a real number $x$; ${(x)_l} = x (x + 1)  \ldots  (x + l - 1)$ is the Pochhammer symbol, for all $l \in \mathbb{Z}^+$.
\end{thm}
\begin{proof}
The proof of Eqs. \eqref{eq:dercheb5}-\eqref{eq:dercheb4} is straightforward using Eqs. \eqref{eq:chebpoly1}-\eqref{eq:chebpolycoeff1}. We prove Eqs. \eqref{eq:dercheb1} and \eqref{eq:dercheb2} by mathematical induction. So consider the case where $k > m \ge 0 \wedge x \ne 0$. For $m = 0$, we can rewrite Eq. \eqref{eq:chebpoly12} with a bit of manipulation in the form
\begin{equation}
{T_k}(x) = \sum\limits_{l = 0}^{\left\lfloor {k/2} \right\rfloor } {c_l^{(k)}{x^{k - 2l}}}.
\end{equation}
For $m=1$, we have 
\[\sum\limits_{l = 0}^{\left\lfloor {k/2} \right\rfloor } {\gamma _{l,k}^{(1)}{\mkern 1mu} {\kern 1pt} c_l^{(k)}{x^{k - 2l - 1}}}  = \sum\limits_{l = 0}^{\left\lfloor {k/2} \right\rfloor } {(k - 2l)c_l^{(k)}{x^{k - 2l - 1}}}  = \frac{d}{{dx}}\sum\limits_{l = 0}^{\left\lfloor {k/2} \right\rfloor } {c_l^{(k)}{x^{k - 2l}}}  = {T'_k}(x),\]
so the theorem is true for $m = 0$ and $1$. Now, assume that the theorem is true for $m = n$, then for $m = n + 1$, we have
\[\sum\limits_{l = 0}^{\left\lfloor {k/2} \right\rfloor } {\gamma _{l,k}^{(n + 1)}\,c_l^{(k)}{x^{k - 2l - n - 1}}}  = \sum\limits_{l = 0}^{\left\lfloor {k/2} \right\rfloor } {(k - 2l - n)\,\gamma _{l,k}^{(n)}\,c_l^{(k)}{x^{k - 2l - n - 1}}}  = \frac{d}{{dx}}\sum\limits_{l = 0}^{\left\lfloor {k/2} \right\rfloor } {\gamma _{l,k}^{(n)}\,c_l^{(k)}{x^{k - 2l - n}}}  = T_k^{(n + 1)}(x).\]
Hence Eq. \eqref{eq:dercheb1} is true for every positive integer $m$. Now consider the case $k > m \ge 0 \wedge x = 0$. The proof of Eq. \eqref{eq:dercheb2} for $m = 0$ is trivial, so consider the case $m = 1$, where the proof is derived as follows:
\[\beta _k^{(1)}\cos \left( {\frac{\pi }{2}(k - 1)} \right) = {( - 1)^2}k\cos \left( {\frac{\pi }{2}(k - 1)} \right) = k\sin \left( {\frac{{k\pi }}{2}} \right) = {T'_k}(0).\]
Now assume that Eq. \eqref{eq:dercheb2} is true for $m = n$. We show that it is also true for $m = n + 1$. Since
\begin{equation}\label{chebp1}
    T_k^{(n + 1)}(x) = \frac{{{d^n}}}{{d{x^n}}}{{T'}_k}(x),
\end{equation}
then substituting Eq. \eqref{eq:chebp2} in Eq. \eqref{chebp1} yields
\begin{equation}\label{eq:chebp3}
T_k^{(n + 1)}(x) = \frac{k}{2}\frac{{{d^n}}}{{d{x^n}}}(\phi (x)\;\psi (x)),
\end{equation}
where $\phi (x) = {T_{k - 1}}(x) - {T_{k + 1}}(x)\;;\;\psi (x) = 1/(1 - {x^2})$. Since
\begin{equation}\label{eq:cheb5}
{\psi ^{(n)}}(0) = \left\{ {\begin{array}{*{20}{l}}
{0,\quad n\;{\text{is odd}},}\\
{n!,\quad n\;{\text{is even}},}
\end{array}} \right.
\end{equation}
then by the general Leibniz rule,
\begin{equation}\label{chebp4}
{(\phi (x) \cdot \psi (x))^{(n)}} = \sum\limits_{k = 0}^n {\left( \begin{array}{c}
n\\
k
\end{array} \right)} \,{\phi ^{(n - k)}}(x)\,{\psi ^{(k)}}(x),
\end{equation}
Eq. \eqref{eq:chebp3} at $x=0$ is reduced to
\begin{align*}
	T_k^{(n + 1)}(0) &= \frac{k}{2}\sum\limits_{\scriptstyle l = 0\atop
	\scriptstyle l\;{\text{is even}}}^n {\left( {\begin{array}{*{20}{l}}
	n\\
	l
	\end{array}} \right)\,{\phi ^{(n - l)}}(0)\,{\psi ^{(l)}}(0)}  = \frac{1}{2}\, k \cdot n!\;\sum\limits_{\scriptstyle l = 0 \atop
	\scriptstyle l\;{\text{is even}}}^n {\frac{{{\phi ^{(n - l)}}(0)}}{{(n - l)!}}}  = \frac{1}{2}\, k \cdot n!\;\sum\limits_{\scriptstyle l = 0 \atop
	\scriptstyle l\;{\text{is even}}}^n {\frac{1}{{(n - l)!}}\left( {T_{k - 1}^{(n - l)}(0) - T_{k + 1}^{(n - l)}(0)} \right)}\\
	&= \frac{1}{2}\,k \cdot n!\;\sum\limits_{\scriptstyle l = 0 \atop
\scriptstyle l\;{\text{is even}}}^n {\frac{1}{{(n - l)!}}\left( {\beta _{k - 1}^{(n - l)}\cos \left( {\frac{\pi }{2}\left( {k - 1 - {\delta _{\frac{{n - l + 1}}{2},\left\lfloor {\frac{{n - l + 1}}{2}} \right\rfloor }}} \right)} \right) - \beta _{k + 1}^{(n - l)}\cos \left( {\frac{\pi }{2}\left( {k + 1 - {\delta _{\frac{{n - l + 1}}{2},\left\lfloor {\frac{{n - l + 1}}{2}} \right\rfloor }}} \right)} \right)} \right)}\\
& = \beta _k^{(n + 1)}\cos \left( {\frac{\pi }{2}\left( {k - {\delta _{\frac{{n + 2}}{2},[\frac{{n + 2}}{2}]}}} \right)} \right),
\end{align*}
which completes the proof.
\end{proof}
Introducing the parameters, 
\begin{equation}\label{eq:paramtheta1}
	{\theta _j} = \left\{ \begin{array}{l}
	1/2,\quad j = 0,n,\\
	1,\quad j = 1, \ldots ,n - 1,
	\end{array} \right.
\end{equation}
replaces formulas (\ref{eq:clenshaw1}) and (\ref{coeff}) with
\begin{equation}\label{clenshawmod}
    P_n(x) = \sum\limits_{k = 0}^n {\theta _k a_k T_k (x)},
\end{equation}
\begin{equation}\label{eq:coeff2}
a_k  = \frac{2}{n}\sum\limits_{j = 0}^n {\theta _j f_j T_k (x_j )},
\end{equation}
where $f_j = {f({x_j})}\; \forall j$. Substituting Eq. \eqref{eq:coeff2} into Eq. \eqref{clenshawmod} yields
\begin{equation}\label{eq:inter}
P_n(x) = \frac{2}{n} \sum\limits_{k = 0}^n {\sum\limits_{j = 0}^n {\theta _j \theta _k f_j T_k (x_j ) T_k (x)} }.
\end{equation}

The derivatives of the Chebyshev interpolant $P_n (x)$ of any order $m$ are then computed at the CGL points by differentiating \eqref{eq:inter} such that
 \begin{equation}
	 P_n^{(m)} (x_i ) = \frac{2}{n} \sum\limits_{j = 0}^n \sum\limits_{k = 0}^n \theta _j \theta _k f_j T_k (x_j )T_k^{(m)} (x_i )  = \sum\limits_{j = 0}^n {d_{ij}^{(m)} f_j }, \quad m \ge 0,
\end{equation}
where 
\begin{equation}
	d_{ij}^{(m)}  = {\frac{2 \theta _j}{n} \sum\limits_{k = 0}^n \theta _k T_k (x_j ) T_k^{(m)} (x_i )},
\end{equation}
are the elements of the $m$th-order CPSDM. With the aid of Theorem \ref{thm1}, we can now calculate the elements of the $m$th-order CPSDM using the following useful formula:

\begin{subequations}\label{eq:CPSDMAT1}
\begin{align}
	d_{ij}^{(m)} &= \frac{2 {\theta _j}}{n}\mathop \sum \limits_{k = m}^n {\theta _k}{T_k}({x_j})\left( {\left\{ \begin{array}{l}
	1,\quad k = m = 0,\\
	{2^{k - 1}}m!,\quad k = m \ge 1,\\
	\mathop \sum \limits_{l = 0}^{\left\lfloor {k/2} \right\rfloor } c_l^{(k)}\,\gamma _{l,k}^{(m)}\,x_i^{k - 2l - m},\quad k > m \ge 0 \wedge {x_i} \ne 0,\\
	\beta _k^{(m)}\cos \left( {\frac{\pi }{2}\left( {k - {\delta _{\frac{{m + 1}}{2}, \lfloor \frac{{m + 1}}{2} \rfloor }}} \right)} \right),\quad k > m \ge 0 \wedge {x_i} = 0
	\end{array} \right.} \right)\\
	&= \frac{2 {\theta _j}}{n}\mathop \sum \limits_{k = m}^n {\theta _k}{T_k}({x_j})\left( {\left\{ \begin{array}{l}
	1,\quad k = m = 0,\\
	{2^{k - 1}}m!,\quad k = m \ge 1,\\
	\mathop \sum \limits_{l = 0}^{\left\lfloor {k/2} \right\rfloor } c_l^{(k)}\,\gamma _{l,k}^{(m)}\,x_i^{k - 2l - m},\quad k > m \ge 0 \wedge i \ne \frac{n}{2},\\
	\beta _k^{(m)}\cos \left( {\frac{\pi }{2}\left( {k - {\delta _{\frac{{m + 1}}{2}, \lfloor \frac{{m + 1}}{2} \rfloor }}} \right)} \right),\quad k > m \ge 0 \wedge i = \frac{n}{2}
	\end{array} \right.} \right).\label{eq:CPSDMAT1b}
\end{align}\end{subequations}

Using the following periodic property of the cosine function
\begin{equation}
	\cos \left( {\pi x} \right) = {\left( { - 1} \right)^{ \lfloor x \rfloor }}\cos \left( {\pi \left( {x -  \lfloor x \rfloor } \right)} \right)\;\forall x \in \mathbb{R},
\end{equation}
we can further rewrite Eqs. \eqref{eq:CPSDMAT1b} as follows:
\begin{equation}\label{eq:finaldiffm1}
	d_{ij}^{(m)} = \frac{2 {\theta _j}}{n}\mathop \sum \limits_{k = m}^n {\theta _k}{\left( { - 1} \right)^{ \lfloor jk/n \rfloor }}{x_{jk - n{\kern 1pt} \left\lfloor {\frac{{jk}}{n}} \right\rfloor }}\left( {\left\{ \begin{array}{l}
	1,\quad k = m = 0,\\
	{2^{k - 1}}m!,\quad k = m \ge 1,\\
	\mathop \sum \limits_{l = 0}^{ \lfloor k/2 \rfloor } c_l^{(k)}\gamma _{l,k}^{(m)}x_i^{k - 2l - m},\quad k > m \ge 0 \wedge i \ne \frac{n}{2},\\
	{\left( { - 1} \right)^{\lfloor \sigma _k^{(m)} \rfloor}}\beta _k^{(m)}{x_{n\left( {\sigma _k^{(m)} -  \lfloor \sigma _k^{(m)} \rfloor } \right)}},\quad k > m \ge 0 \wedge i = \frac{n}{2}
	\end{array} \right.} \right),
\end{equation}
where $\sigma _k^{(m)} = \left( {k - {\delta _{\frac{{1 + m}}{2},\left\lfloor {\frac{{1 + m}}{2}} \right\rfloor }}} \right)/2$. To improve the accuracy of Eqs. \eqref{eq:finaldiffm1}, we can use the negative sum trick, computing all the off-diagonal elements then applying the formula
\begin{equation}\label{eq:sum}
    d_{ii}^{(m)}  =  - \sum\limits_{\scriptstyle j = 0 \hfill \atop
  \scriptstyle j \ne i \hfill}^n {d_{ij}^{(m)} }\quad \forall m \ge 1,
\end{equation}
to compute the diagonal elements. Applying the last trick gives the formula
\begin{equation}\label{eq:sumtrick}
d_{ij}^{(m)} = \left\{ \begin{array}{l}
\frac{{2{\theta _j}}}{n}\sum\limits_{k = m}^n {{\theta _k}} {\left( { - 1} \right)^{ \lfloor jk/n \rfloor }}{x_{jk - n\,\left\lfloor {\frac{{jk}}{n}} \right\rfloor }}\left( {\left\{ \begin{array}{l}
1,\quad k = m = 0 \wedge i \ne j,\\
{2^{k - 1}}m!,\quad k = m \ge 1 \wedge i \ne j,\\
\sum\limits_{l = 0}^{ \lfloor k/2 \rfloor } {c_l^{(k)}} \gamma _{l,k}^{(m)}x_i^{k - 2l - m},\quad k > m \ge 0 \wedge i \ne \frac{n}{2} \wedge i \ne j,\\
{\left( { - 1} \right)^{ \lfloor \sigma _k^{(m)} \rfloor }}\beta _k^{(m)}{x_{n\left( {\sigma _k^{(m)} -  \lfloor \sigma _k^{(m)} \rfloor } \right)}},\quad k > m \ge 0 \wedge i = \frac{n}{2} \wedge i \ne j,
\end{array} \right.} \right)\\
 - \sum\nolimits_{\scriptstyle s = 0\hfill\atop
\scriptstyle s \ne i\hfill}^n {d_{is}^{(m)}},\quad i = j
\end{array} \right..
\end{equation}
Hence, the elements of the first- and second-order CPSDMs are given by
\begin{equation}\label{eq:for31}
d_{ij}^{(1)} = \left\{ \begin{array}{l}
\frac{{2{\theta _j}}}{n}\sum\limits_{k = 1}^n {{\theta _k}} {\left( { - 1} \right)^{ \lfloor jk/n \rfloor }}{x_{jk - n\,\left\lfloor {\frac{{jk}}{n}} \right\rfloor }}\left( {\left\{ \begin{array}{l}
1,\quad k = 1 \wedge i \ne j,\\
\sum\limits_{l = 0}^{ \lfloor k/2 \rfloor } {\left( {k - 2l} \right)} \,c_l^{(k)}x_i^{k - 2l - 1},\quad k > 1 \wedge i \ne \frac{n}{2} \wedge i \ne j,\\
{\left( { - 1} \right)^{ \lfloor \frac{{k - 1}}{2} \rfloor }}k{\mkern 1mu} {x_{n\left( {\frac{{k - 1}}{2} -  \lfloor \frac{{k - 1}}{2} \rfloor } \right)}},\quad k > 1 \wedge i = \frac{n}{2} \wedge i \ne j,
\end{array} \right.} \right)\\
 - \sum\nolimits_{\scriptstyle s = 0\hfill\atop
\scriptstyle s \ne i\hfill}^n {d_{is}^{(1)}} ,\quad i = j
\end{array} \right.,
\end{equation}
%
\begin{equation}\label{eq:for32}
d_{ij}^{(2)} = \left\{ \begin{array}{l}
\frac{{2{\theta _j}}}{n}\sum\limits_{k = 2}^n {{\theta _k}} {\left( { - 1} \right)^{ \lfloor jk/n \rfloor }}{x_{jk - n\,\left\lfloor {\frac{{jk}}{n}} \right\rfloor }}\left( {\left\{ \begin{array}{l}
4,\quad k = 2 \wedge i \ne j,\\
\sum\limits_{l = 0}^{ \lfloor k/2 \rfloor } {\left( {k - 2l - 1} \right)} \left( {k - 2l} \right)c_l^{(k)}x_i^{k - 2l - 2},\quad k > 2 \wedge i \ne \frac{n}{2} \wedge i \ne j,\\
 - {\left( { - 1} \right)^{ \lfloor \frac{k}{2} \rfloor }}{k^2}{\mkern 1mu} {x_{n\left( {\frac{k}{2} - \left\lfloor {\frac{k}{2}} \right\rfloor } \right)}},\quad k > 2 \wedge i = \frac{n}{2} \wedge i \ne j,
\end{array} \right.} \right)\\
 - \sum\nolimits_{\scriptstyle s = 0\hfill\atop
\scriptstyle s \ne i\hfill}^n {d_{is}^{(2)}} ,\quad i = j
\end{array} \right.,
\end{equation}
respectively.

\section{Proposed Line Search Method}
\label{sec:PLSM}
In this section we present a novel line search method that we shall call the Chebyshev PS line search method (CPSLSM). The key idea behind our new approach is five-fold: 
\begin{enumerate}
	\item express the function as a linear combination of Chebyshev polynomials, 
	\item find its derivative, 
	\item find the derivative roots, 
	\item determine a local minimum in $[-1,1]$, and,
	\item finally reverse the change of variables to obtain the approximate local minimum of the original function.
\end{enumerate}
 To describe the proposed method, let us denote by $\mathbb{P}_n$ the space of polynomials of degree at most $n$, and suppose that we want to find a local minimum $t^*$ of a twice-continuously differentiable nonlinear single-variable function $f(t)$ on a fixed interval $[a, b]$ to within a certain accuracy $\varepsilon$. Using the change of variable,
\begin{equation}\label{eq:Kchangeofvar1}
	x = (2\,t - a - b)/(b - a),
\end{equation}
we transform the interval $[a, b]$ into $[-1, 1]$. We shall refer to a point $x$ corresponding to a candidate local minimum point $t$ according to Eq. \eqref{eq:Kchangeofvar1} by the translated candidate local minimum point.

Now let $I_4 f \in \mathbb{P}_4$, be the fourth-degree Chebyshev interpolant of $f$ at the CGL points such that
\begin{equation}\label{eq:Chebintext1}
	f(x;a,b) \approx {I_4}f(x;a,b) = \sum\limits_{k = 0}^4 {{{\tilde f}_k}\,{T_k}(x)}.
\end{equation}
We can determine the Chebyshev coefficients $\{{{\tilde f}_k}\}_{k=0}^4$ via the discrete Chebyshev transform
\begin{align}\label{eq:chebcoeffakin1}
{{\tilde f}_k} &= \frac{1}{{2\,{c_k}}}\sum\limits_{j = 0}^4 {\frac{1}{{{c_j}}}\cos \left( {\frac{{k\,j\,\pi }}{4}} \right)\,{}_a^b{f_j}}\nonumber \\
 &= \frac{1}{{2\,{c_k}}}\left[ {\frac{1}{2}\left( {{}_a^b{f_0} + {{( - 1)}^k}\,{}_a^b{f_4}} \right) + \sum\limits_{j = 1}^3 {\frac{1}{{{c_j}}}\cos \left( {\frac{{k\,j\,\pi }}{4}} \right)\,{}_a^b{f_j}} } \right],
\end{align}
where ${}_a^b{f_j} = f({x_j};a,b),\;j = 0, \ldots ,4,$ and
\begin{equation}
{c_k} = \left\{ \begin{array}{l}
	2,\quad k = 0,4,\\
	1,\quad k = 1, \ldots ,3.
	\end{array} \right.
\end{equation}
We approximate the derivative of $f$ by the derivative of its interpolant ${I_4}f(x;a,b)$,
\begin{equation}\label{eq:I4dash1}
	{I'_4}f(x;a,b) = \sum\limits_{k = 0}^4 {\tilde f_k^{(1)}{\mkern 1mu} {T_k}(x)} ,
\end{equation}
where the coefficients $\left\{\tilde f_k^{(1)}\right\}_{k=0}^4$ are akin to the coefficients of the original function, $\tilde f_k$, by the following recursion \cite{Kopriva2009}
\begin{equation}\label{eq:ckmain1}
	{\tilde c_k}\,\tilde f_k^{(1)} = \tilde f_{k + 2}^{(1)} + 2\,(k + 1)\,{\tilde f_{k + 1}},\quad k \ge 0,
\end{equation}
\begin{equation}\label{eq:Chebdercoeff1}
{{\tilde c}_k} = \;\left\{ {\begin{array}{*{20}{l}}
{{c_k},\quad k = 0, \ldots ,3,}\\
{1,\quad k = 4.}
\end{array}} \right.
\end{equation}
We can calculate $\left\{\tilde f_k^{(1)}\right\}_{k=0}^4$ efficiently using Algorithm \ref{sec:alg1Chebcoeffder1}. Now we collect the terms involving the same powers of $x$ in Eq. \eqref{eq:I4dash1} to get the cubic algebraic equation
\begin{equation}\label{eq:I4dash1simple1}
	{I'_4}f(x;a,b) = {A_1}{\mkern 1mu} {x^3} + {A_2}{\mkern 1mu} {x^2} + {A_3}{\mkern 1mu} x + {A_4},
\end{equation}
where
\begin{subequations}\label{eqs:subcoeffmaink1}
\begin{align}
{A_1} &= 4{\mkern 1mu} \tilde f_3^{(1)},\label{eqs:subcoeffmaink1a}\\
{A_2} &= 2{\mkern 1mu} \tilde f_2^{(1)},\label{eqs:subcoeffmaink1b}\\
{A_3} &= \tilde f_1^{(1)} - 3{\mkern 1mu} \tilde f_3^{(1)},\label{eqs:subcoeffmaink1c}\\
{A_4} &= \tilde f_0^{(1)} - \tilde f_2^{(1)}.\label{eqs:subcoeffmaink1d}
\end{align}
\end{subequations}
Let ${\varepsilon}_{c}$ and $\varepsilon_{\text{mach}}$ denote a relatively small positive number and the machine precision that is approximately equals $2.2204 \times 10^{-16}$ in double precision arithmetic, respectively. To find a local minimum of $f$, we consider the following three cases:

\textbf{Case} $\bm{1}$: If $\left| {{A_1}} \right|, \left| {{A_2}} \right| < {\varepsilon}_{c}$, then ${I'_4}f(x;a,b)$ is linear or nearly linear. Notice also that $A_3$ cannot be zero, since $f$ is nonlinear and formula \eqref{eq:Chebintext1} is exact for all polynomials $h_n(t) \in \mathbb{P}_4$. This motivates us to simply calculate the root $\bar x =  - {A_4}/{A_3}$, and set 
\begin{equation}\label{eq:verynew1}
	{t^*} \approx \frac{1}{2}\left( {(b - a)\,\bar x + a + b} \right),
\end{equation}
if $\left| \bar x \right| \le 1$. If not, then we carry out one iteration of the golden section search method on the interval $[a, b]$ to determine a smaller interval $[a_1, b_1]$ with candidate local minimum $\tilde t$. If the length of the new interval is below $\varepsilon$, we set ${t^*} \approx \tilde t$, and stop. Otherwise, we calculate the the first- and second-order derivatives of $f$ at the translated point $\tilde x_1$ in the interval $[-1,1]$ defined by
\begin{equation}\label{eq:newpt1}
\tilde x_1 = (2\,\tilde t - a_1 - b_1)/(b_1 - a_1).\\
\end{equation}
To this end, we construct the row CPSDMs $\bm{D}^{(1)} = \left(d_j^{(1)}\right)$ and $\bm{D}^{(2)} = \left(d_j^{(2)}\right)$ of length $(m+1)$, for some $m \in \mathbb{Z}^+$ using the following formulas:
\begin{equation}\label{eq:for311}
d_j^{(1)} = \left\{ \begin{array}{l}
\frac{{2{\theta _j}}}{m}\sum\limits_{k = 1}^m {{\theta _k}} {\left( { - 1} \right)^{ \lfloor jk/m \rfloor }}{x_{jk - m\,\left\lfloor {\frac{{jk}}{m}} \right\rfloor }}\left( {\left\{ \begin{array}{l}
1,\quad k = 1 \wedge j \ne m,\\
\sum\limits_{l = 0}^{ \lfloor k/2 \rfloor } {\left( {k - 2l} \right)} c_l^{(k)}{{\tilde x}_1}^{k - 2l - 1},\quad k > 1 \wedge {{\tilde x}_1} \ne 0 \wedge j \ne m,\\
{\left( { - 1} \right)^{ \lfloor \frac{{k - 1}}{2} \rfloor }}k{\mkern 1mu} {x_{m\left( {\frac{{k - 1}}{2} -  \lfloor \frac{{k - 1}}{2} \rfloor } \right)}},\quad k > 1 \wedge {{\tilde x}_1} = 0 \wedge j \ne m,
\end{array} \right.} \right)\\
 - \sum\nolimits_{\scriptstyle s = 0}^{m-1} {d_{s}^{(1)}} ,\quad j = m
\end{array} \right.,
\end{equation}
\begin{equation}\label{eq:for322}
d_j^{(2)} = \left\{ \begin{array}{l}
\frac{{2{\theta _j}}}{m}\sum\limits_{k = 2}^m {{\theta _k}} {\left( { - 1} \right)^{ \lfloor jk/m \rfloor }}{x_{jk - m\,\left\lfloor {\frac{{jk}}{m}} \right\rfloor }}\left( {\left\{ \begin{array}{l}
4,\quad k = 2 \wedge j \ne m,\\
\sum\limits_{l = 0}^{ \lfloor k/2 \rfloor } {\left( {k - 2l - 1} \right)} \left( {k - 2l} \right)c_l^{(k)}{{\tilde x}_1}^{k - 2l - 2},\quad k > 2 \wedge {{\tilde x}_1} \ne 0 \wedge j \ne m,\\
 - {\left( { - 1} \right)^{ \lfloor \frac{k}{2} \rfloor }}{k^2}{\mkern 1mu} {x_{m\left( {\frac{k}{2} - \left\lfloor {\frac{k}{2}} \right\rfloor } \right)}},\quad k > 2 \wedge {{\tilde x}_1} = 0 \wedge j \ne m,
\end{array} \right.} \right)\\
 - \sum\nolimits_{\scriptstyle s = 0}^{m-1} {d_{s}^{(2)}} ,\quad j = m
\end{array} \right.,
\end{equation}
respectively.	The computation of the derivatives can be carried out easily by multiplying $\bm{D}^{(1)}$ and $\bm{D}^{(2)}$ with the vector of function values; that is, 
\begin{subequations}\label{eq:der12k1}
\begin{align}
\bm{\mathcal{F}'} &\approx {{\mathbf{D}}^{(1)}} \bm{\mathcal{F}},\\
\bm{\mathcal{F}''} &\approx {{\mathbf{D}}^{(2)}} \bm{\mathcal{F}},
\end{align}
\end{subequations}
where $\bm{\mathcal{F}'} = \left[{}_{{a_1}}^{{b_1}}{f'_0}, \ldots, {}_{{a_1}}^{{b_1}}{f'_m}\right]^T; \bm{\mathcal{F}''} = \left[{}_{{a_1}}^{{b_1}}{f''_0}, \ldots, {}_{{a_1}}^{{b_1}}{f''_m}\right]^T$. Notice here how the CPSLSM deals adequately with Drawback (i) of \cite{Elgindy2008a} by calculating only one row from each of the CPSDMs in each iterate. If ${{\mathbf{D}}^{(2)}} \bm{\mathcal{F}} > \varepsilon_{\text{mach}}$, then Newton's direction is a descent direction, and we follow the \cite{Elgindy2008a} approach by updating $\tilde x_1$ according to the following formula
\begin{equation}\label{eq:increm1}
	\tilde x_2 = \tilde x_1 - \frac{{{{\mathbf{D}}^{(1)}}\;\bm{{\mathcal F}}}}{{{{\mathbf{D}}^{(2)}}\;\bm{{\mathcal F}}}}.
\end{equation}
At this stage, we consider the following scenarios:
\begin{description}
	\item[(i)] If the stopping criterion
\begin{equation}\label{eq:stopp1k1}
	\left| {{{\tilde x}_2} - {{\tilde x}_1}} \right| \le {\varepsilon _x} = \frac{{2\varepsilon }}{{{b_1} - {a_1}}},
\end{equation}
is fulfilled, we set 
\[{t^*} \approx (({b_1} - {a_1})\, {{\tilde x}_2} + {a_1} + {b_1})/2,\]
and stop. 
\item[(ii)] If $\left| {{{\tilde x}_2}} \right| > 1,$ we repeat the procedure again starting from the construction of a fourth-degree Chebyshev interpolant of $f(x;a_1,b_1)$ using the CGL points. 
\item[(iii)] If the magnitudes of both ${{\mathbf{D}}^{(1)}} \bm{\mathcal{F}}$ and ${{\mathbf{D}}^{(2)}} \bm{\mathcal{F}}$ are too small, which may appear as we mentioned earlier when the profile of the function $f$ is too flat near the current point, or if the function has a multiple local minimum, then the convergence rate of the iterative scheme \eqref{eq:increm1} is no longer quadratic, but rather linear. We therefore suggest here to apply Brent's method. To reduce the length of the search interval though, we consider the following two cases:
\begin{itemize}
	\item If ${{\tilde x}_2} > {{\tilde x}_1}$, then we carry out Brent's method on the interval $[(({b_1} - {a_1}){{\tilde x}_1} + {a_1} + {b_1})/2,{b_1}]$. Notice that the direction from ${{\tilde x}_1}$ into ${{\tilde x}_2}$ is a descent direction 
	as shown by \cite{Elgindy2008a}. 
	\item If ${{\tilde x}_2} < {{\tilde x}_1}$, then we carry out Brent's method on the interval $[a_1, (({b_1} - {a_1}){{\tilde x}_1} + {a_1} + {b_1})/2]$. 
\end{itemize}
\item[(iv)] If none of the above three scenarios occur, we compute the first- and second-order derivatives of the interpolant at $\tilde x_2$ as discussed before, set $\tilde x_1 := \tilde x_2$, and repeat the iterative formula \eqref{eq:increm1}.
\end{description}
If ${{\mathbf{D}}^{(2)}} \bm{\mathcal{F}} \le \varepsilon_{\text{mach}}$, we repeat the procedure again.\\

\textbf{Case} $\bm{2}$: If $\left| {{A_1}} \right| < {\varepsilon}_{c} \wedge \left|{{A_2}}\right| \ge {\varepsilon}_{c}$, then ${I'_4}f(x;a,b)$ is quadratic such that the second derivative of the derivative interpolant is positive and its graph is concave up (simply convex and shaped like a parabola open upward) if ${{A_2}} \ge {\varepsilon}_{c}$; otherwise, the second derivative of the derivative interpolant is negative and its graph is concave down; that is, shaped like a parabola open downward. For both scenarios, we repeat the steps mentioned in \textbf{Case} $\bm{1}$ starting from the golden section search method.

\textbf{Case} $\bm{3}$: If $\left| {{A_1}} \right| \ge {\varepsilon}_{c}$, then the derivative of the Chebyshev interpolant is cubic. To avoid overflows, we scale the coefficients of ${I'_4}f(x;a,b)$ by dividing each coefficient with the coefficient of largest magnitude if the magnitude of any of the coefficients is larger than unity. This procedure ensures that 
\begin{equation}\label{eq:ineqcoeffscaling1}
	{\max _{1 \le j \le 4}}\left| {{A_j}} \right| \le 1.
\end{equation}
The next step is divided into two subcases:

\textbf{Subcase} $\bm{I}$: If any of the three roots $\{\bar x_i\}_{i=1}^3$ of the cubic polynomial ${I'_4}f(x;a,b)$, is a complex number, or the magnitude of any of them is larger than unity, we perform one iteration of the golden section search method on the interval $[a, b]$ to determine a smaller interval $[a_1, b_1]$ with candidate local minimum $\tilde t$. Again, and as we showed before in \textbf{Case} $\bm{1}$, if the length of the new interval is below $\varepsilon$, we set ${t^*} \approx \tilde t$, and stop. Otherwise, we calculate the translated point $\tilde x_1$ using Eq. \eqref{eq:newpt1}. 

\textbf{Subcase} $\bm{II}$: If all of the roots are real, distinct, and lie within the interval $[-1, 1]$, we find the root that minimizes the value of $f$ among all three roots; that is, we calculate,	
	\begin{equation}\label{eq:bestrootkimo11}
{{\tilde x}_1} = \mathop {\arg \min }\limits_{1 \le i \le 3} f\left( {{{\bar x}_i};a,b} \right).		
	\end{equation}
We then update both rows of the CPSDMs $\bm{D}^{(1)} = \left(d_j^{(1)}\right)$ and $\bm{D}^{(2)} = \left(d_j^{(2)}\right)$ using Eqs. \eqref{eq:for311} and \eqref{eq:for322}, and calculate the first- and second- derivatives of the Chebyshev interpolant using Eqs. \eqref{eq:der12k1}. If ${{\mathbf{D}}^{(2)}} \bm{\mathcal{F}} > \varepsilon_{\text{mach}}$, then Newton's direction is a descent direction, and we follow the same procedure presented in \textbf{Case} $\bm{1}$ except when $\left| {{{\tilde x}_2}} \right| > 1$. To update the uncertainty interval $[a, b]$, we determine the second best root among all three roots; that is, we determine,
\begin{equation}\label{eq:secbest2}
	{{\tilde x}_2} = \mathop {\arg \min }\limits_{1 \le i \le 3} f\left( {{{\bar x}_i};a,b} \right):{{\tilde x}_2} \ne {{\tilde x}_1}.
\end{equation}
Now if ${{\tilde x}_1} > {{\tilde x}_2}$, we replace $a$ with $((b - a) {{\tilde x}_2} + a + b)/2$. Otherwise, we replace $b$ with $((b - a) {{\tilde x}_2} + a + b)/2$. The method then proceeds repeatedly until it converges to the local minimum $t^*$, or the number of iterations exceeds a preassigned value, say $k_{\max}$.

It is noteworthy to mention that the three roots of the cubic polynomial, ${I'_4}f(x)$, can be exactly calculated using the trigonometric method due to Fran\c{c}ois Vi\`{e}te; cf. \cite{Nickalls2006}. In particular, let
\begin{align}
	p &= \frac{{{A_3}}}{{{A_1}}} - \frac{1}{3}{\left( {\frac{{{A_2}}}{{{A_1}}}} \right)^2},\\
q &= \frac{2}{{27}}{\left( {\frac{{{A_2}}}{{{A_1}}}} \right)^3} - \frac{{{A_2}{A_3}}}{{3A_1^2}} + \frac{{{A_4}}}{{{A_1}}},
\end{align}
and define,
\begin{equation}
	C(p,q) = 2\sqrt { - \frac{p}{3}} \cos \left( {\frac{1}{3}{{\cos }^{ - 1}}\left( {\frac{{3q}}{{2p}}\sqrt {\frac{{ - 3}}{p}} } \right)} \right).
\end{equation}
Then the three roots can be easily computed using the following useful formulas
\begin{equation}
	\bar x_i = \bar t_i - \frac{{{A_2}}}{{3{A_1}}},\quad i = 1,2,3,
\end{equation}
where
\begin{subequations}
\begin{align}
{{\bar t}_1} &= C(p,q),\\
{{\bar t}_3} &=  - C(p, - q),\\
{{\bar t}_2} &=  - {{\bar t}_1} - {{\bar t}_3}.
\end{align}
\end{subequations}
If the three roots $\{\bar x_i\}_{i=1}^3$ are real and distinct, then it can be shown that they satisfy the inequalities $\bar x_1 > \bar x_2 > \bar x_3$. 
\begin{rem}
The tolerance ${\varepsilon}_x$ is chosen to satisfy the stopping criterion with respect to the variable $t$, since 
\[\left| {{{\tilde x}_2} - {{\tilde x}_1}} \right| \le {\varepsilon _x} \Rightarrow \left| {{{\tilde t}_2} - {{\tilde t}_1}} \right| \le \varepsilon,\]
where
\begin{equation}\label{eq:tildeti}
	{\tilde t_i} = (({b_1} - {a_1})\,{{\tilde x}_i} + {a_1} + {b_1})/2,\quad i = 1,2.
\end{equation}
\end{rem}
\begin{rem}
To reduce the round-off errors in the calculation of $\tilde x_2$ through Eq. \eqref{eq:increm1}, we prefer to scale the vector of function values $\bm{{\mathcal F}}$ if any of its elements is large. That is, we choose a maximum value ${{\mathcal F}}_{\max}$, and set $\bm{{\mathcal F}}:= \bm{{\mathcal F}}/{\max _{0 \le j \le m}}\left| {{}_{{a_1}}^{{b_1}}{f_j}} \right|$ if ${\max _{0 \le j \le m}}\left| {{}_{{a_1}}^{{b_1}}{f_j}} \right| > {{\mathcal F}}_{\max}$. This procedure does not alter the value of $\tilde x_2$, since the scaling of $\bm{{\mathcal F}}$ is canceled out through division.
\end{rem}
\begin{rem}
The derivative of the Chebyshev interpolant, ${I'_4}f(x)$, has all three simple zeros in $(-1,1)$ if $\sum\nolimits_{k = 0}^3 {\tilde f_k^{(1)}\,{x^k}}$, has all its zeros in $(-1,1)$; cf. \cite{Peherstorfer1995}.
\end{rem}
\begin{rem}
Notice how the CPSLSM handles Drawback (ii) of \cite{Elgindy2008a} by simply estimating an initial guess within the uncertainty interval with the aid of a Chebyshev interpolant instead of calculating the first and second derivatives of the objective function at a population of candidate points; thus reducing the required calculations significantly, especially if several uncertainty intervals were encountered during the implementation of the algorithm due to the presence of the local minimum outside the initial uncertainty interval. Moreover, the CPSLSM handles Drawback (iii) of \cite{Elgindy2008a} by taking advantage of approximate derivative information derived from the constructed Chebyshev interpolant instead of estimating an initial guess by comparing the objective function values at the CGL population points that are distributed along the interval $[-1; 1]$; thus significantly accelerating the implementation of the algorithm. Drawback (iv) is resolved by constraining all of the translated candidate local minima to lie within the Chebyshev polynomials feasible domain $[-1,1]$ at all iterations. Drawback (v) is treated efficiently by combining the popular numerical optimization algorithms: the golden-section algorithm and Newton's iterative scheme endowed with CPSDMs. Brent's method is integrated within the CPSLSM to overcome Drawback (vi).
\end{rem}
The CPSLSM can be implemented efficiently using Algorithms \ref{sec1:alg:CPSLSM2}--\ref{sec1:alg:ChebyshevNewton} in Appendix \ref{appendix:PS1}.
\subsection{Locating an Uncertainty Interval}
\label{eq:laui1}
It is important to mention that the proposed CPSLSM can easily work if the uncertainty interval is not known a priori. In this case the user inputs any initial interval, say $\left[{\tilde a},{\tilde b}\right]$. We can then divide the interval into some $l$ uniform subintervals using $(l+1)$ equally-spaced nodes $\{t_i\}_{i=0}^l$. We then evaluate the function $f$ at those points and find the point $t_j$ that minimizes $f$ such that
\[t_j = \mathop {\arg \min }\limits_{0 \le i \le l} f({t_i}).\]
Now we have the following three cases:
\begin{itemize}
  \item If $0 < j < l$, then we set ${\tilde a} = t_{j-1}$ and ${\tilde b} = t_{j+1}$, and return.
	\item If $j = 0$, then we divide ${\tilde a}$ by $\rho^k$ if ${\tilde a} > 0$, where $\rho \approx 1.618$ is the golden ratio and $k$ is the iteration number as shown by \cite{Elgindy2008a}. However, to avoid Drawback (vii) in Section \ref{sec:TEAHOMR}, we replace the calculated $\tilde a$ with $-1/{\tilde a}$ if ${\tilde a} < 1$. Otherwise, we multiply ${\tilde a}$ by $\rho^k$. In both cases we set ${\tilde b} = t_1$, and repeat the search procedure.
	\item If $j = l$, then we multiply ${\tilde b}$ by $\rho^k$ if ${\tilde b} > 0$. Otherwise, we divide ${\tilde b}$ by $\rho^k$ and replace the calculated ${\tilde b}$ with $-1/{\tilde b}$ if ${\tilde b} > -1$. In both cases we set ${\tilde a} = t_{l-1}$, and repeat the search procedure.
\end{itemize}
The above procedure avoids Drawback (vii), and proceeds repeatedly until an uncertainty interval $[a,b]$ is located, or the number of iterations exceeds $k_{\max}$.
\subsection{The CPSLSM Using First-Order Information Only}
\label{eq:tlsmufoio1}
Suppose that we want to find a local minimum $t^*$ of a differentiable nonlinear single-variable function $f(t)$ on a fixed interval $[a, b]$ to within a certain accuracy $\varepsilon$. Moreover, suppose that the second-order information is not available. We can slightly modify the method presented in Section \ref{sec:PLSM} to work in this case. In particular, in \textbf{Case} $\bm{1}$, we carry out one iteration of the golden section search method on the interval $[a, b]$ to determine a smaller interval $[a_1, b_1]$ with two candidate local minima $\tilde t_1$ and $\tilde t_2$: $f\left(\tilde t_2\right) < f\left(\tilde t_1\right)$. If the length of the new interval is below $\varepsilon$, we set ${t^*} \approx \tilde t_2$, and stop. Otherwise, we calculate the first-order derivatives of $f$ at the two points $\tilde x_1$ and $\tilde x_2$ defined by Eqs. \eqref{eq:tildeti}. We then calculate
\begin{equation}\label{eq:Secantkk1}
	s_1 = \frac{{{{\tilde x}_2} - {{\tilde x}_1}}}{{{\mathbf{D}}_2^{(1)}\bm{{\mathcal F}} - {\mathbf{D}}_1^{(1)}\bm{{\mathcal F}}}},
\end{equation}
where ${\mathbf{D}}_1^{(1)}$ and ${\mathbf{D}}_2^{(1)}$ are the first-order CPSDMs corresponding to the points $\tilde x_1$ and $\tilde x_2$, respectively. If $s_1 > \varepsilon_{\text{mach}}$, we follow \cite{Elgindy2008a} approach, and calculate the secant search direction
\begin{equation}\label{eq:increm111}
	s_2 = - s_1\, \cdot {\mathbf{D}}_2^{(1)}\bm{{\mathcal F}},
\end{equation}
then update $\tilde x_2$ according to the following formula,
\begin{equation}\label{eq:increm11}
	\tilde x_3 = \tilde x_2 + s_2.
\end{equation}
Again, we consider the following course of events:
\begin{description}
	\item[(i)] If the stopping criterion
\[\left| {{{\tilde x}_3} - {{\tilde x}_2}} \right| \le {\varepsilon}_x,\]
is fulfilled, we set 
\[{t^*} \approx (({b_1} - {a_1})\,{{\tilde x}_3} + {a_1} + {b_1})/2,\]
and stop. 
\item[(ii)] If $\left| {{{\tilde x}_3}} \right| > 1,$ we repeat the procedure again starting from the construction of a fourth-degree Chebyshev interpolant of $f(x;a_1,b_1)$ at the CGL points. 
\item[(iii)] The third scenario appears when the magnitudes of both $s_2$ and $1/s_1$ are too small. We then apply Brent's method as described by the following two cases:
\begin{itemize}
	\item If ${{\tilde x}_3} > {{\tilde x}_2}$, then we carry out Brent's method on the interval $[(({b_1} - {a_1}){{\tilde x}_2} + {a_1} + {b_1})/2,{b_1}]$. 
	\item If ${{\tilde x}_3} < {{\tilde x}_2}$, then we carry out Brent's method on the interval $[a_1, (({b_1} - {a_1}){{\tilde x}_2} + {a_1} + {b_1})/2]$. 
\end{itemize}
\item[(iv)] If none of the above three scenarios appear, we replace ${\mathbf{D}}_1^{(1)}\bm{{\mathcal F}}$ by ${\mathbf{D}}_2^{(1)}\bm{{\mathcal F}}$, calculate the first-order derivative of the interpolant at $\tilde x_3$, update $\{s_i\}_{i=1}^2$, set $\tilde x_2 := \tilde x_3$, and repeat the iterative formula \eqref{eq:increm11}. 
\end{description}
Finally, if $s_1 \le \varepsilon_{\text{mach}}$, we repeat the procedure again.

In \textbf{Case} $\bm{3}$ (\textbf{Subcase} $\bm{I}$), we apply one iteration of the golden section search method on the interval $[a, b]$ to determine a smaller interval $[a_1, b_1]$ with candidate local minima $\tilde t_1$ and $\tilde t_2: f\left(\tilde t_2\right) < f\left(\tilde t_1\right)$. If the length of the new interval is below $\varepsilon$, we set ${t^*} \approx \tilde t_2$, and stop. Otherwise, we calculate the two points $\tilde x_1$ and $\tilde x_2$ using Eqs. \eqref{eq:tildeti}. For \textbf{Subcase} $\bm{II}$, we find the best root, ${{\tilde x}_2}$, that minimizes the value of $f$ among all three roots. Then suppose that ${\tilde x_2} = {\bar x_J}$, for some $J = 1, 2, 3$. To calculate the secant search direction, we need another point ${{\tilde x}_1}$ within the interval $[-1,1]$. This can be easily resolved by making use of the useful inequalities $\bar x_1 > \bar x_2 > \bar x_3$. In particular, we proceed as follows:
\begin{itemize}
	\item If $J = 1$, then $x_J > x_2 > x_3$, and we set ${{\tilde x}_1} = {{\tilde x}_2} - ({{\tilde x}_2} - {{\bar x}_2})/{\rho ^2}$.
	\item If $J = 2$, then $x_1 > x_J > x_3$, and we set ${{\tilde x}_1} = {{\tilde x}_2} - ({{\tilde x}_2} - {{\bar x}_3})/{\rho ^2}$.
	\item If $J = 3$, then $x_1 > x_2 > x_J$, and we set ${{\tilde x}_1} = {{\tilde x}_2} + ({{\bar x}_2} - {{\tilde x}_2})/{\rho ^2}$.	
\end{itemize}
This procedure is then followed by updating both rows of the CPSDMs, $\bm{D}_1^{(1)}$ and $\bm{D}_2^{(1)}$, and calculating the first-order derivatives of the Chebyshev interpolant at the two points $\{\tilde x_i\}_{i=1}^2$. We then update $s_1$ using Eq. \eqref{eq:Secantkk1}. If $s_1 > \varepsilon_{\text{mach}}$, then the secant direction is a descent direction, and we follow the same procedure discussed before. The method then proceeds repeatedly until it converges to the local minimum $t^*$, or the number of iterations exceeds the preassigned value $k_{\max}$. Notice that the convergence rate of the CPSLSM using first-order information only is expected to be slower than its partner using second-order information, since it performs the secant iterative formula as one of its ingredients rather than Newton's iterative scheme; thus the convergence rate degenerates from quadratic to superlinear.
\begin{rem}
It is inadvisable to assign the value of ${{\tilde x}_1}$ to the second best root that minimizes the value of $f$, since the function profile could be changing rapidly near ${{\tilde t}_2}$; thus $1/s_1$ yields a poor approximation to the second derivative of $f$. In fact, applying this procedure on the rapidly varying function $f_7$ (see Section \ref{sec:numex1}) near $t = 0$ using the CPSLSM gives the poor approximate solution $t^* \approx 5.4$ with $f_7(t^*) \approx -0.03$.
\end{rem}
\subsection{Integration With Multivariate Nonlinear Optimization Algorithms}
\label{eq:IWMNOA1}
The proposed CPSLSM can be integrated easily with multivariate nonlinear optimization algorithms. Consider, for instance, one of the most popular quasi-Newton algorithms for solving unconstrained nonlinear optimization problems widely known as Broyden-Fletcher-Goldfarb-Shanno (BFGS) algorithm. Algorithm \ref{sec1:alg:MBFGSA1} implements a modified BFGS method endowed with a line search method, where a scaling of the search direction vector, $\bm{p}_k$, is applied at each iteration whenever its size exceeds a prescribed value ${p_{{\max}}}$. This step is required to avoid multiplication with large numbers; thus maintaining the stability of the numerical optimization scheme. Practically, the initial approximate Hessian matrix $\mathbf{B}_{0}$ can be initialized with the identity matrix $\mathbf{I}$, so that the first step is equivalent to a gradient descent, but further steps are more and more refined by $\mathbf{B}_{k}, k = 1, 2, \ldots$. To update the search direction at each iterate, we can easily calculate the approximate inverse Hessian matrix, ${\mathbf{B}}_k^{ - 1}$, for each $k = 1, 2, \ldots$, by applying the Sherman-Morrison formula \cite{Sherman1949} giving
\begin{equation}
	{\mathbf{B}}_{k + 1}^{ - 1} = {\mathbf{B}}_k^{ - 1} + \frac{{\left({\bm{s}}_k^T{{\bm{y}}_k} + {\bm{y}}_k^T{\mathbf{B}}_k^{ - 1}{{\bm{y}}_k}\right)\left({{\bm{s}}_k}{\bm{s}}_k^T\right)}}{{{{\left({\bm{s}}_k^T{{\bm{y}}_k}\right)}^2}}} - \frac{{{\mathbf{B}}_k^{ - 1}{{\bm{y}}_k}{\bm{s}}_k^T + {{\bm{s}}_k}{\bm{y}}_k^T{\mathbf{B}}_k^{ - 1}}}{{{\bm{s}}_k^T{{\bm{y}}_k}}},
\end{equation}
where
\begin{align}
	{\bm{s}_k} &= {\alpha _k}{\bm{p}_k},\\
	{\bm{y}_k} &= \nabla f\left({\bm{x}^{(k + 1)}}\right) - \nabla f\left({\bm{x}^k}\right),
\end{align}
instead of typically solving the linear system
\begin{equation}
	{{\mathbf{B}}_k}{{\bm{p}}_k} =  - \nabla f\left( {{\bm{x}^{(k + 1)}}} \right),\quad k = 1, 2, \ldots.
\end{equation}
Since $\bm{p}_k$ is a descent direction at each iteration, we need to adjust the CPSLSM to search for an approximate local minimum $\alpha_k$ in $\mathbb{R}^+$. To this end, we choose a relatively small positive number, say $\hat \varepsilon$, 
and allow the initial uncertainty interval $[\hat \varepsilon, b]: b > \hat \varepsilon$, to expand as discussed before, but only rightward the real line of numbers. 

\section{Error Analysis}
\label{sec:EA1}
A major step in implementing the proposed CPSLSM lies in the interpolation of the objective function $f$ using the CGL points. In fact, the exactness of formula \eqref{eq:Chebintext1} for all polynomials $h_n(t) \in \mathbb{P}_4$ allows for faster convergence rates than the standard quadratic and cubic interpolation methods. 
From another point of view, the CGL points have a number of pleasing advantages as one of the most commonly used node distribution in spectral methods and numerical discretizations. They include the two endpoints, $-1$ and $1$, so they cover the whole search interval $[-1,1]$. Moreover, it is well known that the Lebesgue constant gives an idea of how good the interpolant of a function is in comparison with the best polynomial approximation of the function. Using Theorem 3.4 in \cite{Hesthaven1998a}, we can easily deduce that the Lebesgue constant, $\Lambda _4^{CGL}$, for interpolation using the CGL set $\{x_i\}_{i=0}^4$, is uniformly bounded by those obtained using the Gauss nodal set that is close to that of the optimal canonical nodal set. In particular, $\Lambda _4^{CGL}$ is bounded by
\begin{equation}
	\Lambda _4^{CGL} < \frac{{\gamma \log (100) + \log (1024/{\pi ^2})}}{{\pi \log (10)}} + {\alpha _3} \approx 1.00918 + {\alpha _3},\\
\end{equation}
 where $\gamma = 0.57721566\ldots$, represents Euler's constant, and $0 < {\alpha _3} < \pi /1152$. 
\subsection{Rounding Error Analysis for the Calculation of the CPSDMs}
\label{subsec:REA1}
In this section we address the effect of round-off errors encountered in the calculation of the elements $d_{01}^{(1)}$ and $d_{01}^{(2)}$ given by 
\begin{align}
d_{0,1}^{(1)} &= \frac{2}{n}{\mkern 1mu} \left( {{{\left( { - 1} \right)}^{ \lfloor 1/n \rfloor }}{x_{1 - n \lfloor \frac{1}{n} \rfloor }} + \sum\limits_{k = 2}^n {{\theta _k}\,{{\left( { - 1} \right)}^{ \lfloor k/n \rfloor }}{x_{k - n{\kern 1pt}  \lfloor \frac{k}{n} \rfloor }}{\mkern 1mu} \sum\limits_{l = 0}^{ \lfloor k/2 \rfloor } {\left( {k - 2l} \right)} \,C_l^{(k)}} } \right),\nonumber\\
&= \frac{2}{n}{\mkern 1mu} \left( {{{\left( { - 1} \right)}^{ \lfloor 1/n \rfloor }}{x_{1 - n \lfloor \frac{1}{n} \rfloor }} + \sum\limits_{k = 2}^{n - 1} {{x_{k - n{\kern 1pt}  \lfloor \frac{k}{n} \rfloor }}{\mkern 1mu} \sum\limits_{l = 0}^{ \lfloor k/2 \rfloor } {\left( {k - 2l} \right)} \,C_l^{(k)}} } \right) - n,\quad n \ge 2,\\
d_{0,1}^{(2)} &= \frac{2}{n}{\mkern 1mu} \left( {4{{\left( { - 1} \right)}^{ \lfloor 2/n \rfloor }}{x_{2 - n \lfloor \frac{2}{n} \rfloor }} + \sum\limits_{k = 3}^n {{\theta _k}{{\left( { - 1} \right)}^{ \lfloor k/n \rfloor }}{x_{k - n \lfloor \frac{k}{n} \rfloor }}{\mkern 1mu} \sum\limits_{l = 0}^{ \lfloor k/2 \rfloor } {\left( {k - 2l - 1} \right)\left( {k - 2l} \right)C_l^{(k)}} } } \right),\nonumber\\
&= \frac{2}{n}\left( {4{{\left( { - 1} \right)}^{ \lfloor 2/n \rfloor }}{x_{2 - n \lfloor \frac{2}{n} \rfloor }} + \sum\limits_{k = 3}^{n - 1} {{x_{k - n \lfloor \frac{k}{n} \rfloor }}\sum\limits_{l = 0}^{ \lfloor k/2 \rfloor } {\left( {k - 2l - 1} \right)\left( {k - 2l} \right)C_l^{(k)}} } } \right) - \frac{1}{3}n\left( {{n^2} - 1} \right),\quad n \ge 3,
\end{align}
respectively, since they are the major elements with regard to their values. Accordingly, they bear the major error responsibility comparing to other elements. So let $\delta \approx 1.11 \times 10^{-16}$, be the round-off unity in the double-precision floating-point system, and assume that $\{x_k^*\}_{k=0}^n$ are the exact CGL points, $\{x_k\}_{k=0}^n$ are the computed values, and $\{\delta _k\}_{k=0}^n$ are the corresponding round-off errors such that
\begin{equation}\label{ch3:exact}
x_k^*  = x_k  + \delta _k \; \forall k,
\end{equation}
with $\left|{\delta _k }\right| \le \delta\; \forall k$. If we denote the exact elements of $\bm{D}^{(1)}$ and $\bm{D}^{(2)}$ by $d_{i,j}^{(1)*}$ and $d_{i,j}^{(2)*}\; \forall i,j$, respectively, then
\begin{align}
  {{d_{01}^{(1)}}^*  - d_{01}^{(1)} }  &= \frac{2}{n}{\mkern 1mu} \left( {{{\left( { - 1} \right)}^{ \lfloor 1/n \rfloor }}{\delta _{1 - n \lfloor \frac{1}{n} \rfloor }} + \sum\limits_{k = 2}^{n - 1} {{\delta _{k - n{\kern 1pt}  \lfloor \frac{k}{n} \rfloor }}{\mkern 1mu} \sum\limits_{l = 0}^{ \lfloor k/2 \rfloor } {\left( {k - 2l} \right)} \,C_l^{(k)}} } \right) \nonumber\\
  &\le \frac{{2\,\delta }}{n}{\mkern 1mu} \left( {1 + \sum\limits_{k = 2}^{n - 1} {{\mkern 1mu} \sum\limits_{l = 0}^{ \lfloor k/2 \rfloor } {\left( {k - 2l} \right)} \,C_l^{(k)}} } \right) \nonumber\\
  &= \frac{{2\,\delta }}{n}{\mkern 1mu} \left( {1 + \sum\limits_{k = 2}^{n - 1} {{k^2}} } \right) = \frac{\delta }{3}(n - 1)(2\,n - 1) = O\left( {{n^2}\delta } \right).\label{eq:d01upb1}
\end{align}
Moreover,
\begin{align}
{d{_{01}^{(2)}}^*} - d_{01}^{(2)} &= \frac{2}{n}\left( {4{{\left( { - 1} \right)}^{ \lfloor 2/n \rfloor }}{\delta _{2 - n \lfloor \frac{2}{n} \rfloor }} + \sum\limits_{k = 3}^{n - 1} {{\delta _{k - n \lfloor \frac{k}{n} \rfloor }}\sum\limits_{l = 0}^{ \lfloor k/2 \rfloor } {\left( {k - 2l - 1} \right)\left( {k - 2l} \right)C_l^{(k)}} } } \right) \nonumber\\
 &\le \frac{{2\,\delta }}{n}\left( {4 + \sum\limits_{k = 3}^{n - 1} {\sum\limits_{l = 0}^{ \lfloor k/2 \rfloor } {\left( {k - 2l - 1} \right)\left( {k - 2l} \right)C_l^{(k)}} } } \right) \nonumber\\
 &= \frac{{2\,\delta }}{n}\left( {4 + \frac{1}{3}\sum\limits_{k = 3}^{n - 1} {\left( {{k^4} - {k^2}} \right)} } \right) = \frac{\delta }{{15}}\left( { - 2 + 5n - 5{n^3} + 2{n^4}} \right) = O\left({n^4}\delta \right).\label{eq:d01upb2}
\end{align}

\begin{rem}
It is noteworthy to mention here that the round-off error in the calculation of ${d_{01}^{(1)}}$ from the classical Chebyshev differentiation matrix is of order $O(n^4 \delta)$; cf. \cite{Canuto1988,Baltensperger2003}. Hence, Formulas \eqref{eq:for31} are better numerically. Moreover, the upper bounds \eqref{eq:d01upb1} and \eqref{eq:d01upb2} are in agreement with those obtained by \cite{Elbarbary2005}. 
\end{rem}
\section{Sensitivity Analysis}
\label{sec:SA1}
The following theorem highlights the conditioning of a given root $\bar x_i, i=1,2,3,$ with respect to a given coefficient $A_j, j = 1,2,3,4$.
\begin{thm}\label{thm:1}
Let ${I'_4}f(x)$ be the cubic polynomial defined by Eq. \eqref{eq:I4dash1simple1} after scaling the coefficients $A_j, j = 1,2,3,4$, such that Condition \eqref{eq:ineqcoeffscaling1} is satisfied. Suppose also that $\bar x_i$ is a simple root of ${I'_4}f(x)$, for some $i = 1,2,3$. Then the relative condition number, ${\kappa _{i,j}}$, of $\bar x_i$ with respect to $A_j$ is given by
\begin{equation}\label{eq:thm11}
{\kappa _{i,j}} = \frac{1}{2}\left| {\frac{{{A_j}\,\bar x_i^{3 - j}}}{{3\,{A_1}\,{{\bar x}_i} + {A_2}}}} \right|.
\end{equation}
Moreover, ${\kappa _{i,j}}$ is bounded by the following inequality
\begin{equation}\label{eq:inequality11}
{\kappa _{i,j}} \le \frac{1}{2}\frac{1}{{\left| {3\,{A_1}\,{{\bar x}_i} + {A_2}} \right|}}.
\end{equation}
\end{thm}
\begin{proof}
Suppose that the $j$th coefficient $A_j$ of ${I'_4}f(x) = \sum\nolimits_{j = 1}^4 {{A_j}{\mkern 1mu} {x^{4 - j}}}$, is perturbed by an infinitesimal quantity $\delta A_j$, so that the change in the polynomial is $\delta {I'_4}f(x)$. Suppose also that $\delta \bar x_i$ denotes the perturbation in the $i$th root $\bar x_i$ of ${I'_4}f(x)$. Then by the mean value theorem
\begin{equation}
-\delta {A_j}{\bar x_i^{4-j}} = \delta {I'_4}f(\bar x_i) = {I''_4}f(\bar x_i)\,\delta \bar {x_i} \Rightarrow \delta {\bar x_i} = \frac{{ -\delta {A_j}\,\bar x_i^{4-j}}}{{I''_4}f(\bar x_i)}.
\end{equation}
The condition number of finding $\bar x_i$ with respect to perturbations in the single coefficient $A_j$ is therefore
\begin{equation*}
\kappa_{i,j}  = \mathop {\lim }\limits_{\delta  \to 0} \mathop {\sup }\limits_{\left|\delta {A_j} \right| \le \delta} \left( {\frac{{\left| {\delta \bar {x_i}} \right|}}{{\left| {\bar {x_i}} \right|}}/\frac{{\left| {\delta {A_j}} \right|}}{{\left| {{A_j}} \right|}}} \right) = \mathop {\lim }\limits_{\delta  \to 0} \mathop {\sup }\limits_{\left|\delta {A_j} \right| \le \delta} \left( {\frac{{\left| { -\delta {A_j}\,\bar x_i^{4-j}/{I''_4}f(\bar x_i)} \right|}}{{\left| {{\bar x_i}} \right|}}/\frac{{\left| {\delta {A_j}} \right|}}{{\left| {{A_j}} \right|}}} \right).
\end{equation*}
\begin{equation}
\Rightarrow \kappa_{i,j} = \left| {\frac{{{A_j}\,\bar x_i^{3-j}}}{{{I''_4}f(\bar x_i)}}} \right|,
\end{equation}
from which Eq. \eqref{eq:thm11} and inequality \eqref{eq:inequality11} follow.
\end{proof}
\section{Numerical Experiments}
\label{sec:numex1}
In the following two sections we show our numerical experiments for solving two sets of one- and multi-dimensional optimization test problems. All numerical experiments were conducted on a personal laptop equipped with an Intel(R) Core(TM) i7-2670QM CPU with 2.20GHz speed running on a Windows 10 64-bit operating system, and the numerical results were obtained using MATLAB software V. R2014b (8.4.0.150421).
\subsection{One-Dimensional Optimization Test Problems}
\label{subsec:ODOTP1}
In this section, we first apply the CPSLSM using second-order information on the seven test functions $f_i, i = 1, \ldots, 7$, considered earlier by \cite{Elgindy2008a}, in addition to the following five test functions
\begin{align*}
	f_8(t) &= (t-3)^{12} + 3 t^4,\\
	f_9(t) &= \log \left( {{t^2} + 1} \right) + \cosh (t) + 1,\\
	f_{10}(t) &= \log \left( {\tanh \left( {{t^2}} \right) + {e^{ - {t^2}}}} \right),\\
	{f_{11}}(t) &= {\left( {t - 99} \right)^2}\sinh \left( {\frac{1}{{1 + {t^2}}}} \right),\\
	{f_{12}}(t) &= {t^3} + \left( {3.7 + t + {t^2} - {t^3}} \right)\tanh \left( {{{\left( { - 5.5 + t} \right)}^2}} \right).
\end{align*}
The plots of the test functions are shown in Figure \ref{funplots}. The exact local minima and their corresponding optimal function values obtained using MATHEMATICA 9 software accurate to $15$ significant digits precision are shown in Table \ref{sec:numerical:tab:tf1}. All of the results are presented against the widely used MATLAB `fminbnd' optimization solver to assess the accuracy and efficiency of the current work. We present the number of correct digits ${\text{cd}}_n: =  - {\log _{10}}\left| {{t^*} - {{\tilde t}^*}} \right|$, obtained for each test function, where ${\tilde t}^*$ is the approximate solution obtained using the competing line search solvers, the CPSLSM and the fminbnd solver. The CPSLSM was carried out using $m = 12, {{\mathcal F}}_{\max} = 100, {\varepsilon}_c = 10^{-15}, \varepsilon = 10^{-10}, k_{\max} = 100$, and the magnitudes of both ${{\mathbf{D}}^{(1)}} \bm{\mathcal{F}}$ and ${{\mathbf{D}}^{(2)}} \bm{\mathcal{F}}$ were considered too small if their values fell below $10^{-1}$. The fminbnd solver was implemented with the termination tolerance `TolX' set at $10^{-10}$. The starting uncertainty intervals $\{I_j\}_{j=1}^{12}$, for the considered test functions are listed in respective order as follows: $I_1 = [0, 10], I_2 = [0, 20], I_3 = [1, 5], I_4 = [0, 5], I_5 = [1, 20], I_6 = [0.5, 5], I_7 = [-10, 10], I_8 = [0, 10], I_9 = [-5, 5], I_{10} = [-2, 2], I_{11} = [0, 10]$, and $I_{12} = [-10, 10]$. 
\begin{table}[H]
\begin{center} 
\scalebox{0.8}{
\resizebox{\textwidth}{!}{ %
\begin{tabular}{ccc}
\toprule
Function & $t^*$ & Optimal function value \\
\cmidrule(r){1-3}
$f_1(t) = t^4  - 8.5\; t^3  - 31.0625\; t^2  - 7.5\; t + 45$ & $8.27846234384512$ & $-2271.58168119200$\\
$f_2(t) = (t + 2)^2 (t + 4)(t + 5)(t + 8)(t - 16)$ & $12.6791200596419$ & $-4.36333999223710 \times 10^6$ \\
$f_3(t) = e^t  - 3\; t^2,\quad t > 0.5$ & $2.83314789204934$ & $-7.08129358237484$\\
$f_4(t) = \cos(t) + (t - 2)^2$ & $2.35424275822278$ & $-0.580237420623167$\\
$f_5(t) = 3774.522/t + 2.27\; t - 181.529,\quad t > 0$ & $40.7772610902992$ & $3.59976534995851$\\
$f_6(t) = 10.2/t + 6.2\; t^3 ,\quad t > 0$ & $0.860541475570675$ & $15.8040029284830$ \\
$f_7(t) = -1/(1+t^2)$ & $0$ & $-1$\\
$f_8(t) = (t-3)^{12} + 3\,t^4$ & $1.82219977424679$ & $40.2016340135967$\\
$f_9(t) = \log \left( {{t^2} + 1} \right) + \cosh (t) + 1$ & $0$ & $2$\\
$f_{10}(t) = \log \left( {\tanh \left( {{t^2}} \right) + {e^{ - {t^2}}}} \right)$ & $0$ & $0$\\
${f_{11}}(t) = {\left( {t - 99} \right)^2}\sinh \left( {1/\left({1 + {t^2}}\right)} \right)$ & $99$ & $0$\\
${f_{12}}(t) = {t^3} + \left( {3.7 + t + {t^2} - {t^3}} \right)\tanh \left( {{{\left( { - 5.5 + t} \right)}^2}} \right)$ & $-0.5$ & $3.45$\\
\bottomrule
\end{tabular}
}}
\caption{The one-dimensional test functions together with their corresponding local minima and optimal values}
\label{sec:numerical:tab:tf1}
\end{center}
\end{table}
Figure \ref{cdn} shows the ${\text{cd}}_n$ obtained using the CPSLSM and the fminbnd solver for each test function. Clearly, the CPSLSM establishes more accuracy than the fminbnd solver in general with the ability to exceed the required precision in the majority of the tests. Moreover, the CPSLSM was able to find the exact local minimum for $f_7$. The experiments conducted on the test functions $f_5$ and $f_{11}$ are even more interesting, because they manifest the adaptivity of the CPSLSM to locate a search interval bracketing the solution when the latter does not lie within the starting uncertainty interval. Moreover, the CPSLSM is able to determine highly accurate approximate solutions even when the starting uncertainty intervals are far away from the desired solutions. The proposed method is therefore recommended as a general purpose line search method. On the other hand, the fminbnd solver was stuck in the starting search intervals, and failed to locate the solutions. For test function $f_6$, we observe a gain in accuracy in favor of the fminbnd solver. Notice though that the approximate solution $\tilde t^* = 0.86053413225062$, obtained using the CPSLSM in this case yields the function value $15.8040029302092$ that is accurate to $9$ significant digits. Both methods yield the same accuracy for the test function $f_8$.

Figure \ref{iter1} further shows the number of iterations $k$ required by both methods to locate the approximate minima given the stated tolerance $\varepsilon$. The figure conspicuously shows the power of the novel optimization scheme observed in the rapid convergence rate, as the CPSLSM requires about half the iterations number required by the fminbnd solver for several test functions. The gap is even much wider for test functions $f_5, f_7, f_9,$ and $f_{10}$. 

\begin{figure}[ht]
\centering
\includegraphics[scale=0.75]{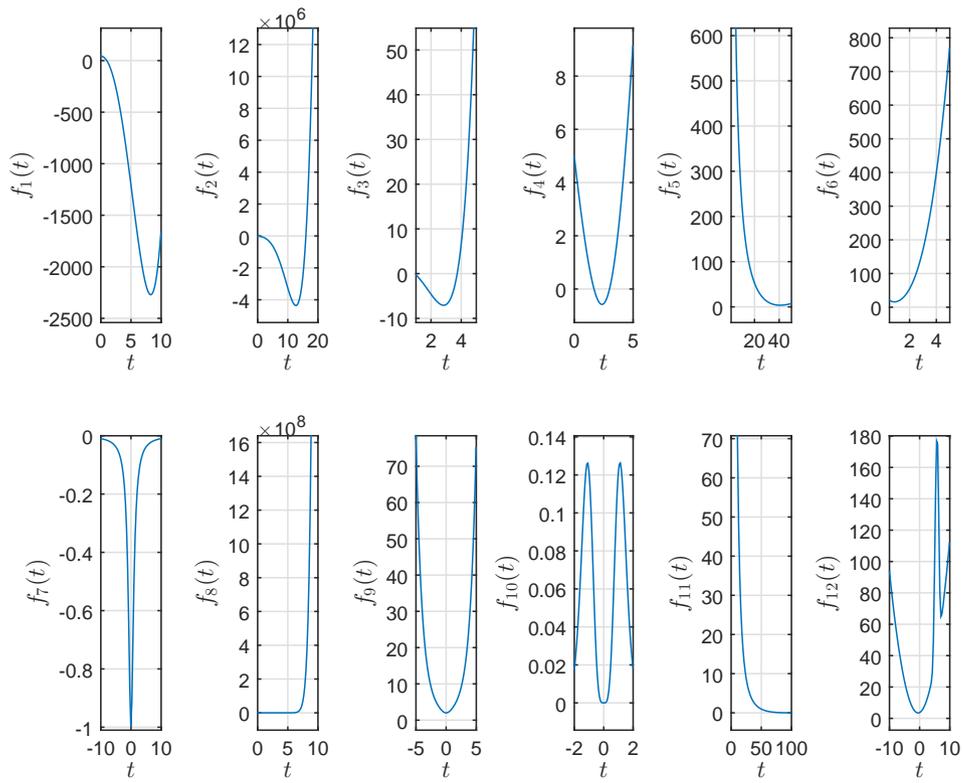}
\caption{The plots of the test functions}
\label{funplots}
\end{figure}
\begin{figure}[ht]
\centering
\includegraphics[scale=0.7]{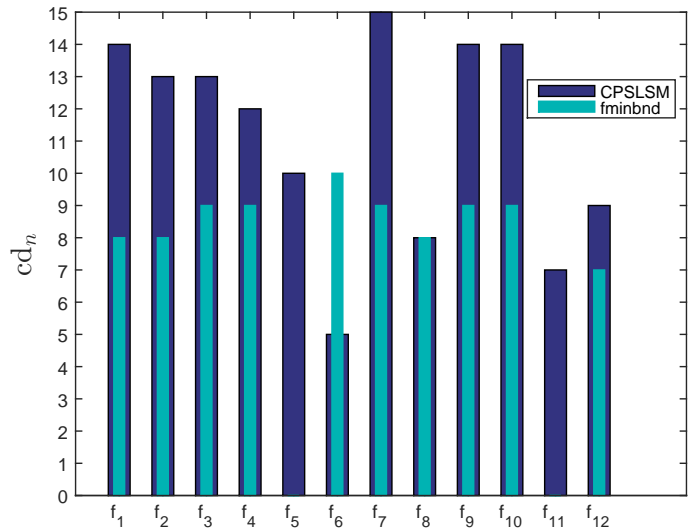}
\caption{The cd$_n$ for the CPSLSM and the fminbnd solver}
\label{cdn}
\end{figure}
\begin{figure}[ht]
\centering
\includegraphics[scale=0.7]{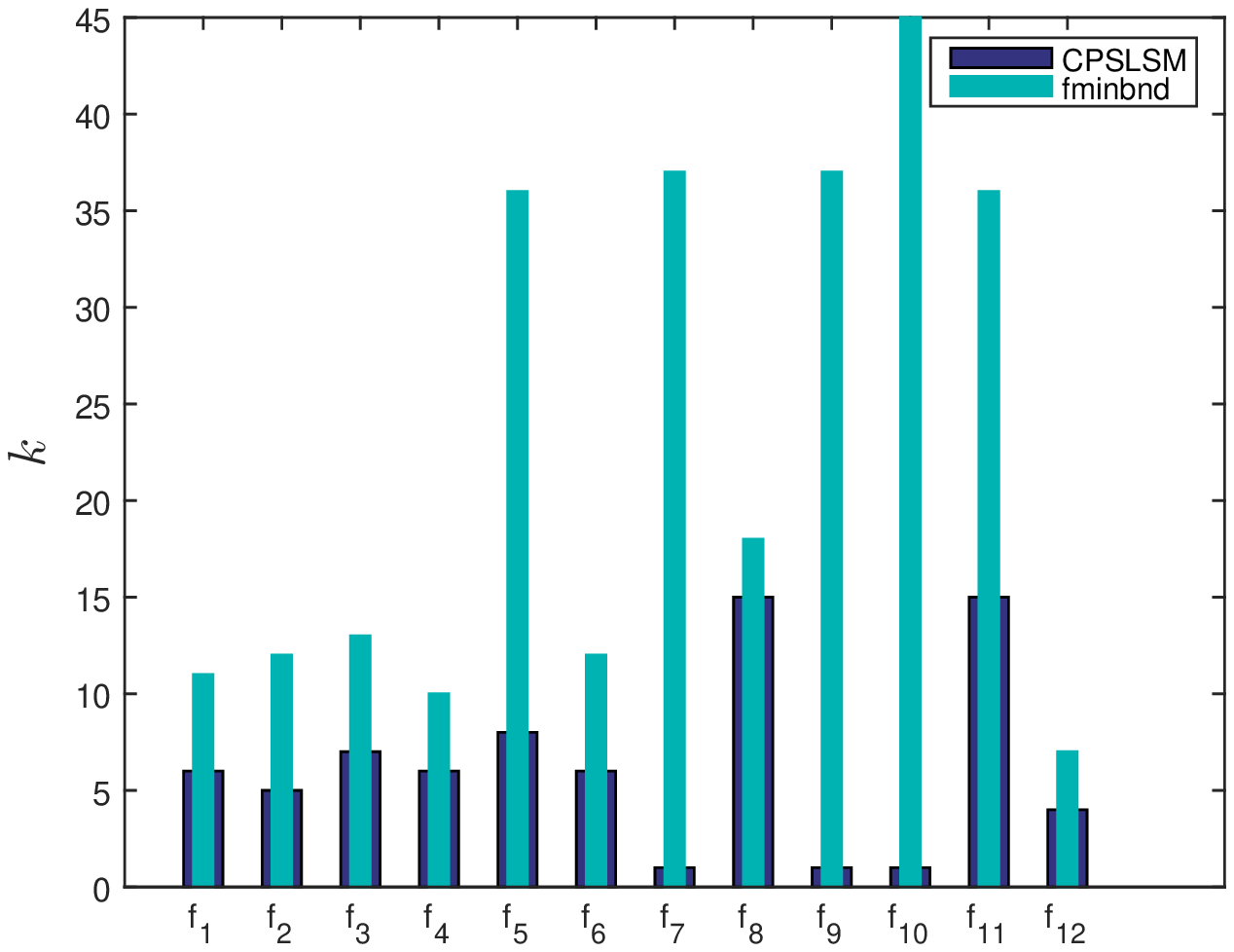}
\caption{The number of iterations required by the CPSLSM and the fminbnd solver}
\label{iter1}
\end{figure}
%

Figures \ref{cdn1} and \ref{iter11storder} show the cd$_n$ and the number of iterations required by the CPSLSM using only first-order information versus the fminbnd solver. Here we notice that the obtained cd$_n$ values using the CPSLSM are almost identical with the values obtained using second-order information with a slight increase in the number of iterations required for some test functions as expected.
\begin{figure}[ht]
\centering
\includegraphics[scale=0.7]{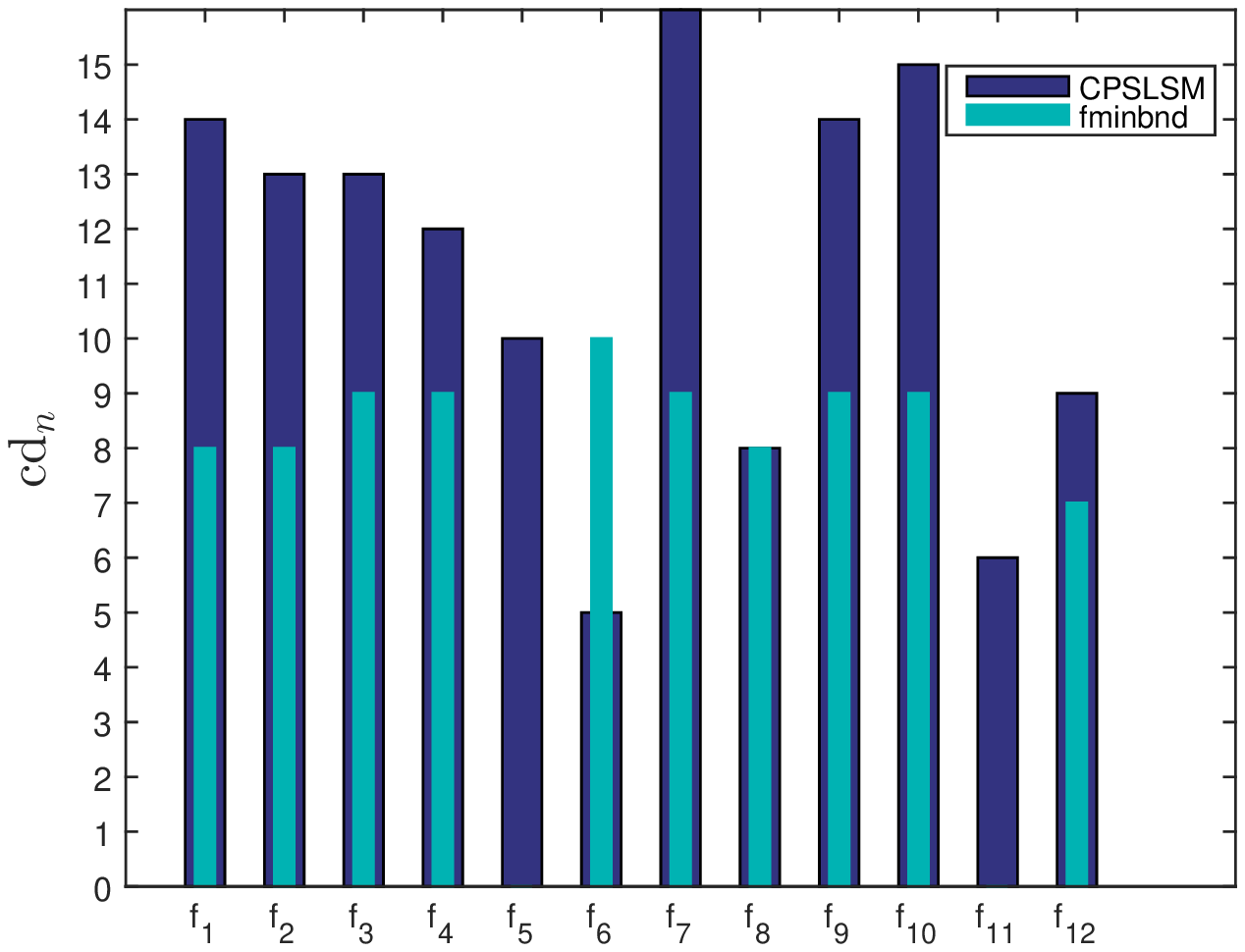}
\caption{The cd$_n$ for the CPSLSM using first-order information and the fminbnd solver}
\label{cdn1}
\end{figure}
\begin{figure}[ht]
\centering
\includegraphics[scale=0.7]{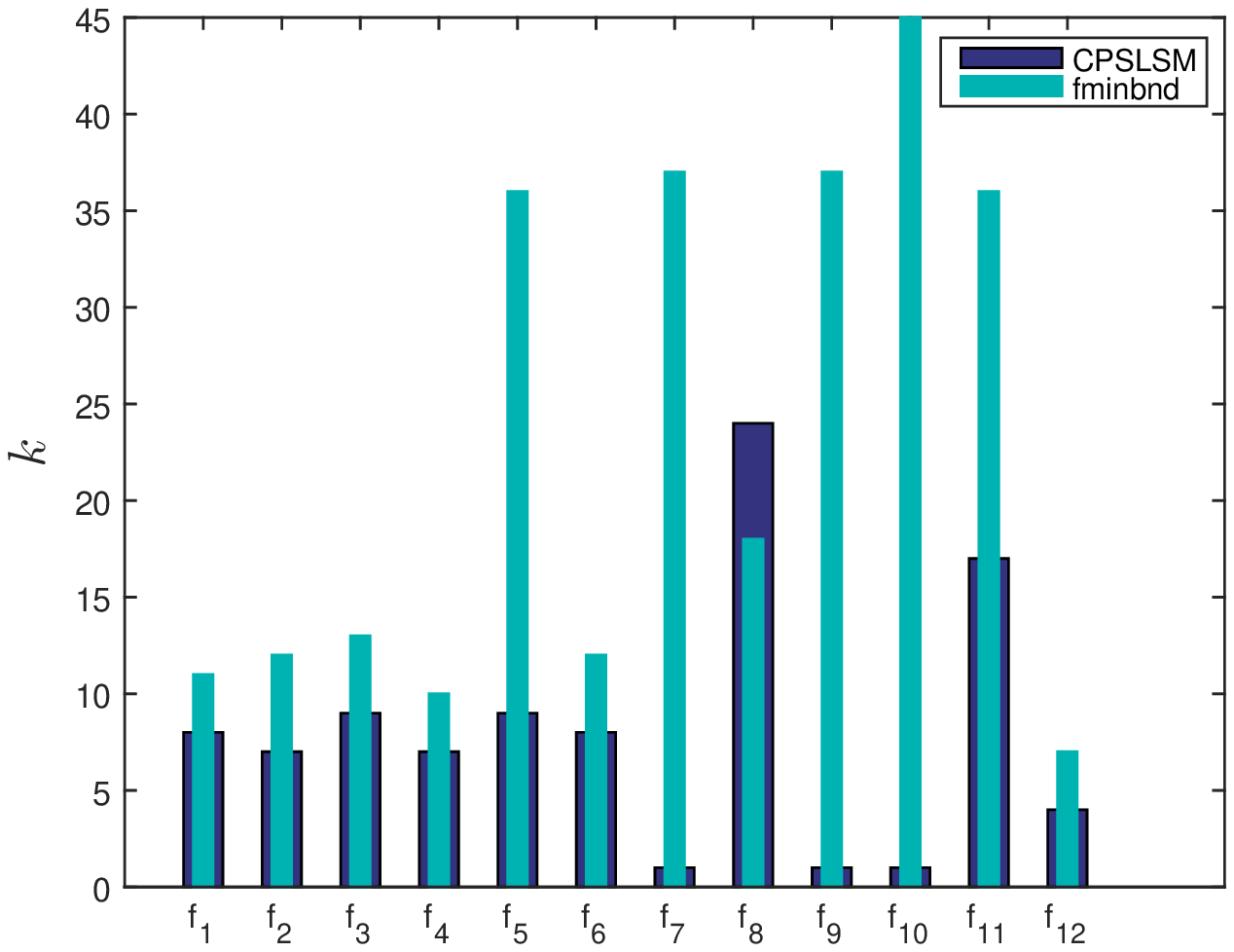}
\caption{The number of iterations required by the CPSLSM using first-order information and the fminbnd solver}
\label{iter11storder}
\end{figure}

\subsection{Multi-Dimensional Optimization Test Problems}
\label{subsec:MDOTP1}
The functions listed below are some of the common functions and data sets used for testing multi-dimensional unconstrained optimization algorithms: 
\begin{itemize}
	\item \textit{Sphere Function}: 
	\[{f_1}(\bm{x}) = \sum\limits_{i = 1}^d {x_i^2},\quad d \in {\mathbb{Z}^ + }.\]
	 The global minimum function value is $f(\bm{x}^*) = 0$, obtained at $\bm{x}^* = [0,\ldots,0]^T$. 
	\item \textit{Bohachevsky Function}: 
	\[{f_2}(\bm{x}) = x_1^2 + 2\,x_2^2 - 0.3\,\cos (3\,\pi {x_1}) - 0.4\,\cos (4\,\pi {x_2}) + 0.7.\]
	 The global minimum function value is $f(\bm{x}^*) = 0$, obtained at $\bm{x}^* = [0,0]^T$. 
	\item \textit{Booth Function}: 
	\[{f_3}(\bm{x}) = {({x_1} + 2\,{x_2} - 7)^2} + {(2\,{x_1} + {x_2} - 5)^2}.\]
	 The global minimum function value is $f(\bm{x}^*) = 0$, obtained at $\bm{x}^* = [1,3]^T$. 
	\item \textit{Three-Hump Camel Function}: 
\[{f_4}(\bm{x}) = 2\,x_1^2 - 1.05\,x_1^4 + x_1^6/6 + {x_1}\,{x_2} + x_2^2.\]
	 The global minimum function value is $f(\bm{x}^*) = 0$, obtained at $\bm{x}^* = [0,0]^T$. 
	\item \textit{Powell Function}: 
\[{f_5}(\bm{x}) = \sum\limits_{i = 1}^{d/4} {\left( {{{({x_{4i - 3}} + 10\,{x_{4i - 2}})}^2} + 5\,{{({x_{4i - 1}} - \,{x_{4i}})}^2} + {{({x_{4i - 2}} - 2\,{x_{4i - 1}})}^4} + 10\,{{({x_{4i - 3}} - \,{x_{4i}})}^4}} \right)},\quad d \in {\mathbb{Z}^ + }. \]
	 The global minimum function value is $f(\bm{x}^*) = 0$, obtained at $\bm{x}^* = [0,\ldots,0]^T$. 
	\item \textit{Goldstein-Price Function}: 
\begin{align*}
{f_6}(\bm{x}) &= \left( {1 + {{({x_1} + {x_2} + 1)}^2}\,\left(19 - 14\,{x_1} + 3\,{x_1}^2 - 14\,{x_2} + 6\,{x_1}\,{x_2} + 3\,x_2^2\right)} \right)\\
 &\times \left( {30 + {{(2\,{x_1} - 3\,{x_2})}^2}\,\left(18 - 32\,{x_1} + 12\,x_1^2 + 48\,{x_2} - 36\,{x_1}\,{x_2} + 27\,x_2^2\right)} \right).
\end{align*}
	 The global minimum function value is $f(\bm{x}^*) = 3$, obtained at $\bm{x}^* = [0,-1]^T$. 
	\item \textit{Styblinski-Tang Function}: 
\[{f_7}(\bm{x}) = \frac{1}{2}\sum\limits_{i = 1}^d {\left( {x_i^4 - 16\,x_i^2 + 5\,{x_i}} \right)},\quad d \in {\mathbb{Z}^ + }.\]
	 The global minimum function value is $f(\bm{x}^*) = -39.16599\,d$, obtained at $\bm{x}^* = [-2.903534,\ldots,-2.903534]^T$. 
	\item \textit{Easom Function}: 
\[{f_8}(\bm{x}) =  - \cos ({x_1})\,\cos ({x_2})\,{e^{ - {{({x_1} - \pi )}^2} - {{({x_2} - \pi )}^2}}}.\]
	 The global minimum function value is $f(\bm{x}^*) = -1$, obtained at $\bm{x}^* = [\pi,\pi]^T$. 
\end{itemize}
Table \ref{sec:numerical:tab:tf2} shows a comparison between the modified BFGS method endowed with MATLAB ``fminbnd'' line search solver (MBFGSFMINBND) and the modified BFGS method integrated with the present CPSLSM (MBFGSCPSLSM) for the multi-dimensional test functions $f_i, i = 1, \ldots, 8$. The modified BFGS method was performed using Algorithm \ref{sec1:alg:MBFGSA1} with $\mathbf{B}_0 = \mathbf{I}, k_{\max} = 10^4$, and $p_{max} = 10$. The gradients of the objective functions were approximated using central difference approximations with step-size $10^{-4}$. Both fminbnd method and CPSLSM were initiated using the uncertainty interval $[\hat \varepsilon,10]$ with $\hat \varepsilon = 3\,\varepsilon_{\text{mach}}$. The maximum number of iterations allowed for each line search method was $100$. Each line search method was considered successful at each iterate $k$ if the change in the approximate step length $\alpha_k$ is below $10^{-6}$. The CPSLSM was carried out using $m = 6, \varepsilon_c = \varepsilon_{\text{mach}}, \varepsilon_D = 10^{-6}$, and ${{\mathcal F}}_{\max} = 100$. Moreover, both MBFGSFMINBND and MBFGSCPSLSM were stopped whenever
\[\left\| {\nabla f\left( {{\bm{x}^{(k)}}} \right)} \right\|_2 < {10^{ - 12}},\]
or
\[\left\| {{\bm{x}^{(k + 1)}} - {\bm{x}^{(k)}}} \right\|_2 < {10^{ - 12}}.\] 
Table \ref{sec:numerical:tab:tf2} clearly manifests the power of the CPSLSM, where the running time and computational cost of the modified BFGS method can be significantly reduced. 
\begin{table}[H]
\begin{center} 
\scalebox{1}{
\resizebox{\textwidth}{!}{ %
\begin{tabular}{ccc}
\toprule
\textbf{Function} & \textbf{MBFGSFMINBND} & \textbf{MBFGSCPSLSM} \\
 $\bm{x}^{(0)}$ & NI/fval/EN/ET & NI/fval/EN/ET\\
 & $\bm{\tilde x}^*$ & $\bm{\tilde x}^*$\\
\cmidrule(r){1-3}
$f_1(\bm{x})\; (d = 4)$ & $13/4.784e-31/6.917e-16/0.021$ & $2/3.3895e-29/5.822e-15/0.011$\\
$[50,1,4,-100]^T$ & $[-3.0928e-16,-6.186e-18,-2.4723e-17,6.1814e-16]^T$ & $[2.6019e-15,5.2024e-17,2.0815e-16,-5.2038e-15]^T$\\\\
$f_1(\bm{x})\; (d = 100)$ & $13/3.8088e-31/6.172e-16/0.047$ & $2/7.3153e-30/2.705e-15/0.013$\\
${[50,1,4,2.5, \ldots ,2.5, - 100]^T}$ & Omitted & Omitted \\\\
$f_2(\bm{x})$ & $412/0.88281/7.766e-01/0.592$ & $16/0.46988/4.695e-01/0.078$ \\
$[10,20]^T$ & $[0.61861,-0.46953]^T$ & $[6.001e-09,0.46953]^T$ \\\\
$f_3(\bm{x})$ & $1/6.3109e-30/1.884e-15/0.002$ & $1/0/0/0.005$ \\
$[2,2]^T$ & $[0.\bar{9},3]^T$ & $[1,3]^T$ \\\\
$f_4(\bm{x})$ & NI exceeded $10000$ (\textbf{Failure}) & $5/1.8396e-32/9.740e-17/0.021$ \\
$[-0.5,1]^T$ & & $[-9.7238e-17,5.6191e-18]^T$ \\\\
$f_5(\bm{x})\; (d = 4)$ & $29/5.4474e-23/2.732e-06/0.051$ & $28/3.6165e-26/4.409e-07/0.141$ \\
$[2,3,1,1]^T$ & $[2.3123e-06,-2.3123e-07,1.0163e-06,1.0163e-06]^T$ & $[3.7073e-07,-3.7073e-08,1.6667e-07,1.6667e-07]^T$ \\\\
$f_6(\bm{x})$ & NI exceeded $10000$ (\textbf{Failure}) & $53/3/9.577e-09/0.291$ \\
$[-0.5, 1]^T$ & & $[5.3617e-09,-1]^T$ \\\\
$f_7(\bm{x})\; (d = 4)$ & $576/-128.39/\text{--}/1.125$ & $11/-128.39/\text{--}/0.052$ \\
$[-4, -4, 5, 5]^T$ & $[-2.9035,-2.9035,2.7468,2.7468]^T$ & $[-2.9035,-2.9035,2.7468,2.7468]^T$ \\\\
$f_7(\bm{x})\; (d = 12)$ & $308/-342.76/\text{--}/0.680$ & $35/-342.76/\text{--}/0.202$ \\
${[3, - 0.5,1.278,1, \ldots 1,0.111,4.5]^T}$ & Omitted & Omitted \\\\
$f_8(\bm{x})$ & NI exceeded $10000$ (\textbf{Failure}) & $3/-1/4.333e-14/0.037$ \\
$[1, 1]^T$ & & $[3.1416,3.1416]^T$ \\
\bottomrule
\end{tabular}
}}
\caption{A comparison between the BFGSFMINBND and the BFGSCPSLSM for the multi-dimensional test functions $f_i, i = 1, \ldots, 8$. ``NI'' denotes the number of iterations required by the modified BFGS algorithm, ``fval'' denotes the approximate minimum function value, ``EN'' denotes the Euclidean norm of the error $\bm{x}^* - \bm{\tilde x}^*$, where $\bm{\tilde x}^*$ is the approximate optimal design vector, and ``ET'' denotes the elapsed time for implementing each  method in seconds. The bar over $9$ is used to indicate that this digit repeats indefinitely}
\label{sec:numerical:tab:tf2}
\end{center}
\end{table}

\section{Conclusion}
\label{Conc}
While there is a widespread belief in the optimization community that an exact line search method is unnecessary, as what we may need is a large step-size which can lead to a sufficient descent in the objective function, the current work largely adheres to the use of adaptive PS exact line searches based on Chebyshev polynomials. The presented CPSLSM is a novel exact line search method that enjoys many useful virtues: (i) The initial guesses of the solution in each new search interval are calculated accurately and efficiently using a high-order Chebyshev PS method. (ii) The function gradient values in each iteration are calculated accurately and efficiently using CPSDMs. On the other hand, typical line search methods in the literature endeavor to approximate such values using finite difference approximations that are highly dependent on the choice of the step-size -- a usual step frequently shared by the optimization community. (iii) The method is adaptive in the sense of locating a search interval bracketing the solution when the latter does not lie within the starting uncertainty interval. (iv) It is able to determine highly accurate approximate solutions even when the starting uncertainty intervals are far away from the desired solution. (v) The accurate approximation to the objective function, quadratic convergence rate to the local minimum, and the relatively inexpensive computation of the derivatives of the function using CPSDMs are some of the many useful features possessed by the method. (vi) The CPSLSM can be efficiently integrated with the current state of the art multivariate optimization methods.  In addition, the presented numerical comparisons with the popular rival Brent's method verify further the effectiveness of the proposed CPSLSM. In general, the CPSLSM is a new competitive method that adds more power to the arsenal of line search methods by significantly reducing the running time and computational cost of multivariate optimization algorithms.
\appendix
\section{Preliminaries}
\label{sec:P1}
In this section, we present some useful results from approximation theory. The first kind Chebyshev polynomial (or simply the Chebyshev polynomial) of degree $n$, $T_n(x)$, is given by the explicit formula
\begin{equation}\label{eq:chebpoly1}
{T_n}(x) = \cos \left( {n\,{{\cos }^{ - 1}}(x)} \right)\;\forall x \in [ - 1,1],
\end{equation}
using trigonometric functions, or in the following explicit polynomial form \cite{Snyder1966}:
\begin{equation}\label{eq:chebpoly12}
{T_n}(x) = \frac{1}{2}\sum\limits_{k = 0}^{\left\lfloor {n/2} \right\rfloor } {T_{n - 2k}^k\,{x^{n - 2k}}} ,
\end{equation}
where
\begin{equation}\label{eq:chebpolycoeff1}
	T_m^k = {2^m}{( - 1)^k}\left\{ {\frac{{m + 2\,k}}{{m + k}}} \right\}\left( \begin{array}{c}
	m + k\\
	k
	\end{array} \right),\quad m,k \ge 0.
\end{equation}
The Chebyshev polynomials can be generated by the three-term recurrence relation
\begin{equation}\label{eq:recur}
T_{n + 1} (x) = 2xT_n (x) - T_{n - 1} (x), \quad n \ge 1,
\end{equation}
starting with ${T_0}(x) = 1$ and ${T_1}(x) = x$. They are orthogonal in the interval $[-1, 1]$ with respect to the weight function $w(x) = \left(1 - x^2\right)^{-1/2}$, and their orthogonality relation is given by
\begin{equation}\label{eq:orthog1}
\left\langle {{T_n},{T_m}} \right\rangle_w  = \int_{ - 1}^1 {{T_n}(x)\,{T_m}(x)\,{{\left(1 - {x^2}\right)}^{ - \frac{1}{2}}}dx = \frac{\pi }{2}{c_n}{\delta _{nm}}} ,
\end{equation}
where $c_0  = 2, c_n  = 1, n \ge 1$, and $\delta _{n m}$ is the Kronecker delta function defined by
\[{\delta _{nm}} = \left\{ {\begin{array}{*{20}{l}}
{1,\quad n = m,}\\
{0,\quad n \ne m.}
\end{array}} \right.\]
The roots (aka Chebyshev-Gauss points) of $T_n(x)$ are given by
\begin{equation}
	{x_k} = \cos \left( {\frac{{2k - 1}}{{2n}}\pi } \right),\quad k = 1, \ldots ,n,
\end{equation}
and the extrema (aka CGL points) are defined by
\begin{equation}\label{eq:app:CGL1}
	{x_k} = \cos \left( {\frac{{k \pi}}{n}} \right),\quad k = 0,1, \ldots ,n.
\end{equation}
The derivative of $T_n(x)$ can be obtained in terms of Chebyshev polynomials as follows \cite{Mason2003}:
\begin{equation}\label{eq:chebp2}
\frac{d}{{dx}}{T_n}(x) = \frac{n}{2}\frac{{{T_{n - 1}}(x) - {T_{n + 1}}(x)}}{{1 - {x^2}}},\quad \left| x \right| \ne 1.
\end{equation}

\cite{Clenshaw1960} showed that a continuous function $f(x)$ with bounded variation on $[-1, 1]$ can be approximated by the truncated series
\begin{equation}\label{eq:clenshaw1}
({P_n}f)(x) = \sum\limits_{k = 0}^n {^{''}}{{a_k}\,{T_k}(x)} ,
\end{equation}
where
\begin{equation}\label{coeff}
{a_k} = \frac{2}{n}\sum\limits_{j = 0}^n {^{''}}{{f_j}\,{T_k}({x_j})} ,\quad n > 0,
\end{equation}
$x_j, j = 0, \ldots, n$, are the CGL points defined by Eq. \eqref{eq:app:CGL1}, $f_j = f(x_j)\, \forall j$, and the summation symbol with double primes denotes a sum with both the first and last terms halved. For a smooth function $f$, the Chebyshev series \eqref{eq:clenshaw1} exhibits exponential convergence faster than any finite power of $1/n$ \cite{Gottlieb1977}. 

%
\section{Pseudocodes of Developed Computational Algorithms}
\label{appendix:PS1}
\begin{algorithm}[ht]
\renewcommand{\thealgorithm}{1}
\caption{The CPSLSM Algorithm Using Second-Order Information}
\label{sec1:alg:CPSLSM2}
\begin{algorithmic}
  \REQUIRE Positive integer number $m$; twice-continuously differentiable nonlinear single-variable objective function $f$; maximum function value ${{\mathcal F}}_{\max}$; uncertainty interval endpoints $a, b$; relatively small positive numbers ${\varepsilon}_{c},\varepsilon_D,\varepsilon$; maximum number of iterations $k_{\max}$.
	\ENSURE The local minimum $t^* \in [a,b]$.
	\STATE{$\rho_1 \leftarrow 1.618033988749895; \rho_2 \leftarrow 2.618033988749895; e^+ = a + b; e^- = b - a; k \leftarrow 0; x_j \leftarrow \cos\left(j \pi/4\right),\quad j = 0,\ldots, 4$.}
	\STATE{Calculate $c_l^{(k)}, l = 0, \ldots ,\left\lfloor {k/2} \right\rfloor, k=0,\ldots,4$ using Eqs. \eqref{eq:clkmain1}; $\{\theta _k,\tilde c_k\}_{k=0}^4$ using Eqs. \eqref{eq:paramtheta1} \& \eqref{eq:Chebdercoeff1}, respectively.}
\IF{$m \ne 4$} \STATE{$x_{D,j} \leftarrow \cos\left(j \pi/m\right),\quad j = 0,\ldots, m$.} \ELSE \STATE{$x_{D,j} \leftarrow x_j ,\quad j = 0,\ldots, 4$.} \ENDIF
\WHILE{$k \le k_{\max}$} 
\STATE{flag $\leftarrow 0; \bm{{{\mathcal F}}_{x}} \leftarrow \left\{ {f\left( {\left( {{e^ - }{x_j} + {e^ + }} \right)/2} \right)} \right\}_{j = 0}^4; {{\mathcal F}}_{x,\max} \leftarrow \|\bm{{{\mathcal F}}_{x}}\|_{\infty}$.
\IF{${{\mathcal F}}_{x,\max} > {{\mathcal F}}_{\max}$} \STATE{$\bm{{{\mathcal F}}_{x}} \leftarrow \bm{{{\mathcal F}}_{x}}/{{\mathcal F}}_{x,\max}$.} \ENDIF
\STATE{Calculate $\left\{{\tilde f}_k, {\tilde f}_k^{(1)}\right\}_{k=0}^4$, using Eqs. \eqref{eq:chebcoeffakin1} and Algorithm \ref{sec:alg1Chebcoeffder1}, and the coefficients $\{A_j\}_{j=1}^2$ using Eqs. \eqref{eqs:subcoeffmaink1a} \& \eqref{eqs:subcoeffmaink1b}, respectively.}
}
\IF{$\left|A_1\right| < {\varepsilon}_{c}$} 
\STATE{Call Algorithm \ref{sec1:alg:loqc}.}
\ELSE 
\STATE{Calculate $\{A_j\}_{j=3}^4$ using Eqs. \eqref{eqs:subcoeffmaink1c} \& \eqref{eqs:subcoeffmaink1d}; $A_{\max} \leftarrow \mathop {\max }\limits_{1 \le j \le 4} \left| {{A_j}} \right|$.\\
\IF{$A_{\max} > 1$} \STATE{$A_j \leftarrow A_j/A_{\max}, \quad j = 1, \ldots, 4$.} \ENDIF
\STATE{Calculate the roots $\{\bar x_i\}_{i=1}^3$.}
\IF{$\bar x_i$ is complex \OR $\left|\bar x_i\right| > 1$ for any $i = 1, 2, 3$} \STATE{$k \leftarrow k + 1$; Call Algorithm \ref{sec1:alg:osgssa1}; $e^+ \leftarrow a + b; \tilde x_1 \leftarrow (2\,\tilde t_1 - e^+)/e^-$; flag $\leftarrow 1$.} \ELSE \STATE{Determine ${\tilde x_1}$ using Eq. \eqref{eq:bestrootkimo11}.} \ENDIF\\
Compute $\bm{\mathcal{F'}}$ and $\bm{\mathcal{F''}}$ at $\tilde x_1$.
\IF{$\bm{\mathcal{F''}} > \varepsilon_{\text{mach}}$} \STATE{Call Algorithm \ref{sec1:alg:ChebyshevNewton}.} \ENDIF
\IF{flag $= 1$} \STATE{\CONTINUE} \ELSE \STATE{Compute $\tilde x_2$ using Eq. \eqref{eq:secbest2}.\\
\IF{$\tilde x_1 > \tilde x_2$} \STATE{$a \leftarrow (e^- \tilde x_2 + e^+)/2$.} \ELSE \STATE{$b \leftarrow (e^- \tilde x_2 + e^+)/2$.} \ENDIF
} \ENDIF\\
$e^+ \leftarrow a + b; e^- \leftarrow b - a; k \leftarrow k + 1$.
}
\ENDIF
\ENDWHILE
\STATE{Output(`Maximum number of iterations exceeded.').}
\RETURN{$t^*$}.
\end{algorithmic}
\end{algorithm}
\begin{algorithm}[ht]
\renewcommand{\thealgorithm}{2}
\caption{Calculating the Chebyshev Coefficients of the Derivative of a Polynomial Interpolant}
\label{sec:alg1Chebcoeffder1}
\begin{algorithmic}
\REQUIRE The Chebyshev coefficients $\{{{\tilde f}_k}\}_{k=0}^4$.
\STATE $\tilde f_{4}^{(1)} \leftarrow 0$.
\STATE $\tilde f_{3}^{(1)} \leftarrow 8\,\tilde f_{4}$.
\STATE{$\tilde f_{k}^{(1)} \leftarrow 2\,(k+1)\,\tilde f_{k+1} + \tilde f_{k+2}^{(1)},\quad k = 2, 1$.}
\STATE{$\tilde f_{0}^{(1)} \leftarrow \tilde f_{1} + \tilde f_{2}^{(1)}/2$.}
\RETURN{$\{\tilde f_{k}^{(1)}\}_{k=0}^4$.}
\end{algorithmic}
\end{algorithm}
\begin{algorithm}[ht]
\renewcommand{\thealgorithm}{3}
\caption{Linear/Quadratic Derivative Interpolant Case}
\label{sec1:alg:loqc}
\begin{algorithmic}
  \REQUIRE $m;f; a; b; e^-; \rho_1; \rho_2; A_2; k; k_{\max}; \{x_{D,j}\}_{j=0}^m; {{\mathcal F}}_{\max}; \varepsilon_c; \varepsilon_D; \varepsilon$.
\IF{$\left|A_2\right| < \varepsilon_c$}
\STATE{Calculate $\{A_j\}_{j=3}^4$ using Eqs. \eqref{eqs:subcoeffmaink1c} \& \eqref{eqs:subcoeffmaink1d}; $\bar x \leftarrow -A_4/A_3$.\\
\IF{$\left|\bar x\right| \le 1$} \STATE {${t^*} \leftarrow \frac{1}{2}\left( {(b - a)\,\bar x + a + b} \right)$; Output($t^*$); Stop.} \ENDIF}
\ENDIF
\STATE{$k \leftarrow k + 1$; Call Algorithm \ref{sec1:alg:osgssa1};\\
 $e^+ \leftarrow a + b; \tilde x_1 \leftarrow (2 \tilde t_1 - e^+)/e^-; \bm{\mathcal{F}} \leftarrow \left\{ {f\left( {\left( {{e^ - }{x_{D,j}} + {e^ + }} \right)/2} \right)} \right\}_{j = 0}^m$; $\mathcal{F}_{1,\max} \leftarrow \|\bm{\mathcal{F}}\|_{\infty}$.\\ 
\IF{$\mathcal{F}_{1,\max} > {{\mathcal F}}_{\max}$} 
\STATE{$\bm{\mathcal{F}} \leftarrow \bm{\mathcal{F}}/\mathcal{F}_{1,\max}$.} \ENDIF\\
Compute $\bm{\mathcal{F}'}$ and $\bm{\mathcal{F}''}$ at $\tilde x_1$ using Eqs. \eqref{eq:for311}--\eqref{eq:der12k1}.\\
\IF{$\bm{\mathcal{F}''} > 0$} \STATE{Call Algorithm \ref{sec1:alg:ChebyshevNewton}.} \ENDIF\\
 \CONTINUE
}
\end{algorithmic}
\end{algorithm}
\begin{algorithm}[ht]
\renewcommand{\thealgorithm}{4}
\caption{One-Step Golden Section Search Algorithm}
\label{sec1:alg:osgssa1}
\begin{algorithmic}
  \REQUIRE $f; a; b; e^-; \rho_1; \rho_2; \varepsilon$.
\STATE{$t_1 \leftarrow a + e^-/\rho_2; t_2 \leftarrow a + e^-/\rho_1$.}
\IF{$f(t_1) < f(t_2)$} \STATE{$b \leftarrow t_2; t_2 \leftarrow t_1; e^- \leftarrow b - a; t_1 \leftarrow a + e^-/\rho_2$.} \ELSE \STATE{$a \leftarrow t_1; t_1 \leftarrow t_2; e^- \leftarrow b - a; t_2 \leftarrow a + e^-/\rho_1$.} \ENDIF
\IF{$f(t_1) < f(t_2)$} \STATE{$\tilde t_1 \leftarrow t_1; b \leftarrow t_2.$} \ELSE \STATE{$\tilde t_1 \leftarrow t_2; a \leftarrow t_1.$} \ENDIF
\STATE{$e^- \leftarrow b - a$.}
\IF{$e^- \le \varepsilon$} \STATE{$t^* \leftarrow \tilde t_1$; Output($t^*$); Stop.} \ELSE \STATE{return.} \ENDIF
\RETURN{$\tilde t_1,a,b,e^-$}.
\end{algorithmic}
\end{algorithm}
\begin{algorithm}[ht]
\renewcommand{\thealgorithm}{5}
\caption{The Chebyshev-Newton Algorithm}
\label{sec1:alg:ChebyshevNewton}
\begin{algorithmic}
  \REQUIRE $m; f; a; b; e^-; e^+; \tilde x_1; k; k_{\max}; \varepsilon_D; \varepsilon; \bm{\mathcal{F}}; \bm{\mathcal{F}'};\bm{\mathcal{F}''}$.
\WHILE{$k \le k_{\max}$} 
\STATE Determine $\tilde x_2$ using Eq. \eqref{eq:increm1}; $k \leftarrow k + 1$.
\IF{Condition \eqref{eq:stopp1k1} is satisfied} 
\STATE{$t^* \leftarrow (e^- \tilde x_2 + e^+)/2$.}
\STATE{Output($t^*$); Stop.} 
\ELSIF{$\left| {{{\tilde x}_2}} \right| > 1$} \STATE{\BREAK} 
\ELSIF{$\left|\bm{\mathcal{F}'}\right| < \varepsilon_D$ \AND $\left|\bm{\mathcal{F}''}\right| < \varepsilon_D$} 
\STATE{\IF{$\tilde x_2 > \tilde x_1$} \STATE{Apply Brent's method on the interval $[(e^- {{\tilde x}_1} + e^+)/2,{b}]$; Output($t^*$); Stop.} \ELSE \STATE{Apply Brent's method on the interval $[a, (e^- {{\tilde x}_1} + e^+)/2]$; Output($t^*$); Stop.} \ENDIF}
\ELSE \STATE{Calculate $\bm{\mathcal{F}'}$ and $\bm{\mathcal{F}''}$ at $\tilde x_2$ using Eqs. \eqref{eq:der12k1}; $\tilde x_1 \leftarrow \tilde x_2$.}
\ENDIF
\ENDWHILE
\RETURN{$k,\tilde x_1$.}
\end{algorithmic}
\end{algorithm}
\begin{algorithm}[ht]
\renewcommand{\thealgorithm}{6}
\caption{Modified BFGS Algorithm With a Line Search Strategy}
\label{sec1:alg:MBFGSA1}
\begin{algorithmic}
  \REQUIRE Objective function $f$; initial guess $\bm{x}^{(0)}$; an approximate Hessian matrix $\mathbf{B}_{0}$; maximum number of iterations $k_{\max}$; maximum direction size ${p_{{\max}}}$.
\STATE $k \leftarrow 0$; calculate $\mathbf{B}_{k}^{-1}$, and the gradient vector $\nabla f\left({\bm{x}^{(k)}}\right)$.
\STATE $\bm{x}^* \leftarrow {\bm{x}^{(k)}}$ if the convergence criterion is satisfied.
\STATE ${\bm{p}_k} \leftarrow  - \nabla f\left({\bm{x}^{(k)}}\right); {p_{{\text{norm}}}} \leftarrow {\left\| {{\bm{p}_k}} \right\|_2}.$ 
\IF{${p_{{\text{norm}}}} > {p_{{\max}}}$} \STATE{${\bm{p}_k} \leftarrow {\bm{p}_k}/{p_{{\text{norm}}}}.$ \COMMENT{Scaling}} \ENDIF
\WHILE{$k \le k_{\max}$} 
\STATE Perform a line search to find an acceptable step-size $\alpha_{k}$ in the direction $\bm{p}_{k}$.
\STATE $\bm{s}_k \leftarrow \alpha_k \bm{p}_{k}$.
\STATE $\bm{x}^{(k+1)}\leftarrow\bm{x}^{(k)}+\alpha_{k} \bm{p}_{k}$. \COMMENT{Update the state vector}
\STATE Calculate $\nabla f\left({\bm{x}^{(k+1)}}\right)$, and set $\bm{x}^* \leftarrow {\bm{x}^{(k+1)}}$ if the convergence criterion is satisfied.
\STATE ${\bm{y}_k} \leftarrow \nabla f\left({\bm{x}^{(k + 1)}}\right) - \nabla f\left({\bm{x}^{(k)}}\right); t \leftarrow \bm{s}_k^T{\bm{y}_k}; \mathbf{T}_1 \leftarrow {\bm{y}_k}\bm{s}_k^T; \mathbf{T}_2 \leftarrow {{\mathbf{B}}_{k}^{ - 1}}{\mathbf{T}_1}$.
\STATE ${\mathbf{B}}_{k + 1}^{ - 1} \leftarrow {\mathbf{B}}_{k}^{ - 1} + {\left( {t + \bm{y}_k^T{{\mathbf{B}}_{k}^{ - 1}}{\bm{y}_k}} \right)\left( {{\bm{s}_k}\bm{s}_k^T} \right)}/t^2 - \left({\mathbf{T}_2 + \mathbf{T}_2^T}\right)/{t}.$
\STATE ${\bm{p}_{k + 1}} \leftarrow - {\mathbf{B}}_{k + 1}^{ - 1}\nabla f\left( {{\bm{x}^{(k + 1)}}} \right); {p_{{\text{norm}}}} \leftarrow {\left\| {{\bm{p}_{k+1}}} \right\|_2}.$ 
\IF{${p_{{\text{norm}}}} > {p_{{\max}}}$} \STATE{${\bm{p}_{k+1}} \leftarrow {\bm{p}_{k+1}}/{p_{{\text{norm}}}}.$ \COMMENT{Scaling}} \ENDIF
\STATE $k \leftarrow k+1$.
\ENDWHILE
\STATE{Output(`Maximum number of iterations exceeded.').}
\RETURN{$\bm{x}^*; f\left(\bm{x}^*\right)$.}
\end{algorithmic}
\end{algorithm}


\end{document}